\documentclass[12pt, reqno]{scrartcl}
\usepackage{authblk}

\usepackage{lscape,amssymb,amsthm, amsmath,verbatim,graphicx,xcolor,url,booktabs,longtable,paralist,natbib,ulem,url}
\usepackage{dsfont}
\makeatletter
\def\NAT@spacechar{~}
\makeatother

\usepackage[textsize=scriptsize]{todonotes}

\usepackage{subfigure}

\usepackage{paralist}

\usepackage{multicol}
\usepackage{epsfig,microtype}

\usepackage[pdftex,colorlinks = true]{hyperref}
\hypersetup{citecolor=darkblue,linkcolor=darkblue,urlcolor=darkblue}
\definecolor{darkblue}{rgb}{0,0,.6}

\usepackage{parskip}

\usepackage{enumerate}

\allowdisplaybreaks

\theoremstyle{plain}
\newtheorem{thm}{Theorem}

\newtheorem{assumption}{Assumption}

\newtheorem{lemma}{Lemma}

\newcommand \eg{\varepsilon}

\theoremstyle{remark}
\newtheorem*{rmrk*}{Remark}
\newtheorem{rmrk}{Remark}

\usepackage[margin=1in]{geometry}

\newcommand{\opnorm}[1]{\vert\kern-0.25ex\vert\kern-0.25ex\vert #1\vert\kern-0.25ex\vert\kern-0.25ex\vert}
\newcommand{\opnormBig}[1]{\Bigl\vert\kern-0.25ex\Bigl\vert\kern-0.25ex\Bigl\vert#1\Bigr\vert\kern-0.25ex\Bigr\vert\kern-0.25ex\Bigr\vert}

\title{A general white noise test based on kernel lag-window estimates of the spectral density operator}
\author[a]{Vaidotas Characiejus}
\author[b]{Gregory Rice}
\affil[a]{D\'epartement de math\'ematique, Universit\'e libre de Bruxelles, Belgium}
\affil[b]{Department of Statistics and Actuarial Science, University of Waterloo, Canada}

\begin{document}







\maketitle

\begin{abstract}
\noindent We propose a general white noise test for functional time series based on estimating a distance between the spectral density operator of a weakly stationary time series and the constant spectral density operator of an uncorrelated time series. The estimator that we propose is based on a kernel lag-window type estimator of the spectral density operator. When the observed time series is a strong white noise in a real separable Hilbert space, we show that the asymptotic distribution of the test statistic is standard normal, and we further show that the test statistic diverges for general serially correlated time series. These results recover as  special cases those of Hong~(1996) and Horv\'ath et al. (2013). In order to implement the test, we propose and study a number of kernel and bandwidth choices, including a new data adaptive bandwidth, as well as data adaptive power transformations of the test statistic that improve the normal approximation in finite samples. A simulation study demonstrated that the proposed method has good size and improved power when compared to other methods available in the literature, while also offering a light computational burden.


\end{abstract}

\section{Introduction}
With this paper, we aim to contribute to the growing literature on general white noise tests for functional time series. The scope of methodology for analyzing time series taking values in a function space has grown substantially over the last two decades; see \cite{bosq:2000} for seminal work on linear processes in function space, and \cite{hormann:kokoszka:2012}, as well as Chapters 13--16 of \cite{HKbook} for more recent reviews.  Most methods are still founded on non-parametric and non-likelihood based approaches that rely on the estimation of autocovariance operators to quantify serial dependence in the sequence; see, e.g., \cite{kargin:onatski:2008, klepsch:kluppelberg:2017, panaretos:tavakoli:2013, zhang:shao:2015}. It is then often of interest to measure whether or not an observed sequence of curves, or model residuals, exhibit significant autocorrelation.

Quite a large number of methods have been introduced in order to test for serial correlation in the setting of functional time series. These methods build upon a well developed literature on portmanteau tests for scalar and multivariate time series, see e.g.  \cite{li:2004} and Chapter 5 of \cite{francq:zakoian:2010} for a summary, and can be grouped along the classical time series dichotomy of time domain and spectral domain approaches. In the time domain, \cite{gabrys:kokoszka:2007} proposed a method based on applying a multivariate portmanteau test to vectors of scores obtained from projecting the functional time series into a few principal directions. \cite{horvath:huskova:rice:2013} and \cite{kokoszka:rice:shang:2017} develop portmanteau tests based directly on the norms of empirical autocovariance operators. In the spectral domain, \cite{zhang:2016} and \cite{bagchi:characiejus:dette:2018} develop test statistics based on measuring the difference between periodogram operators and the spectral density operator under the assumption that the sequence is a weak white noise. Among these, the tests of \cite{horvath:huskova:rice:2013}, \cite{zhang:2016} and \cite{bagchi:characiejus:dette:2018} are general white noise tests in the sense that they are asymptotically consistent against serial correlation at any lag, whereas the other tests cited utilize only the autocovariance information up to a user selected maximum lag, and do not have more than trivial power for serial correlation occurring beyond that lag.

Each of these statistics are indelibly connected in that they reduce to weighted sums of semi-norms of empirical autocovariance operators, and hence it is natural to consider alternate ways of choosing the weights in order to increase power; see \cite{escanciano:lobato:2009} for a nice discussion of these connections for scalar portmanteau tests. The statistics developed in \cite{zhang:2016} and \cite{bagchi:characiejus:dette:2018}, which are based on the periodogram operator, effectively give equal weights to the autocovariance operators at all lags. One might expect then that reweighting the periodogram so that it is a more efficient estimate of the spectral density operator might lead to  more powerful tests. A sensible and well studied method of choosing these weights is to employ a kernel and lag-window or bandwidth as is common in nonparametric spectral density operator estimation. \cite{panaretos:tavakoli:2013} were the first to put forward similar estimators of the spectral density operator based on smoothing the periodogram operator for $L^2([0,1],\mathbb R)$-valued time series. \cite{hong:1996} developed general white noise tests for scalar sequences based on this principle by comparing kernel lag-window estimates of the spectral density to the constant spectral density of a weak white noise, which were further studied and extended in \cite{shao:2011} and \cite{shao:2011:2}.

In this paper, we develop a general white noise test based on kernel lag-window estimators of the spectral density operator of a time series with values in a separable Hilbert space~$\mathbb H$. Under standard conditions on the kernel and bandwidth used, we show that the estimated distance between the spectral density operator and the constant spectral density operator based on such estimators can be normalized to have an asymptotic standard normal distribution when the observed series is a strong white noise. We further show that this standardized distance diverges in probability to infinity at a quantifiable rate under general departures from a weak white noise.

These results compare to and generalize the main results of \cite{horvath:huskova:rice:2013}. Letting $n$ denote the length of the functional time series and $p_n$ denote the bandwidth parameter, the main result of \cite{horvath:huskova:rice:2013} establishes the asymptotic normality of the unweighted sum of the norms of empirical autocovariance operators up to a bandwidth $p_n$ satisfying $p_n\to \infty$, $p_n=O(\log(n))$, as $n\to \infty$. Our main result establishes this under the less restrictive conditions on the bandwidth $p_n\to \infty$, $p_n=o(n)$, while also allowing for a general class of weights in the sum. This condition on the bandwidth is optimal in the sense that without it the corresponding kernel lag-window estimator of the spectral density cannot generally be consistent in mean squared norm sense. These improvements owe primarily to a new, more general, method of proof relying on a martingale central limit theorem.

Implementing the test requires the choice of a kernel and bandwidth, and we suggest several methods for this, including a new data driven and kernel adaptive bandwidth. We also investigate power transformations of our test statistics that improve their size properties in finite samples. We investigated these results and various choices of the kernel and bandwidth by means of a Monte Carlo simulation study, which confirm that the proposed tests have good size as well as power exceeding currently available tests, although with the drawback that they are not built for general white noise (e.g.\ functional conditionally heteroscedastic) series.

The rest of the paper is organized as follows. We present in Section~\ref{sec:main} our main methodological contributions and theory, including all asymptotic results for the proposed test statistic and its power transformation. In Section~\ref{simul}, we discuss some details of implementing the proposed test, and present the results of a Monte Carlo simulation study. The practical utility of the proposed tests are illustrated in Section~\ref{sec:emp} with an application to Eurodollar futures curves. Some concluding remarks are given in Section~\ref{sec:conc}, and all proofs as well as the definition of the data adaptive bandwidth are given in the Appendix following the references.


\section{Statement of method and main results}\label{sec:main}
We first define some notation that is used throughout the paper. Suppose that $\mathbb H$ is a real separable Hilbert space with the inner product $\langle\cdot,\cdot\rangle:\mathbb H\times\mathbb H\to\mathbb R$. For $x,y\in\mathbb H$, the tensor of $x$ and $y$ is a rank one operator $x\otimes y:\mathbb H\to\mathbb H$ defined by $(x\otimes y)(z)=\langle z,y\rangle x$ for each $z\in\mathbb H$.  Let $\{u_t\}_{t\in\mathbb Z}$ be a second-order stationary sequence of $\mathbb H$-valued random elements. Throughout we make use of the following moment conditions:
\begin{assumption}\label{assumption:moment}
 $\operatorname Eu_t=0$ and $\operatorname E\|u_t\|^4<\infty$ for each $t\in\mathbb Z$.
\end{assumption}
The autocovariance operators of  $\{u_t\}_{t\in\mathbb Z}$ are defined by
\begin{equation}\label{eq:dfn_covop}
	\mathcal C(j)
	=\operatorname E[u_j\otimes u_0],
	\quad j\in\mathbb Z.
\end{equation}
The spectral density operator $\mathcal F(\omega)$ with $\omega\in[-\pi,\pi]$ is the discrete-time Fourier transform of the sequence of autocovariance operators $\{\mathcal C(j)\}_{j\in\mathbb Z}$ defined by
\begin{equation}\label{eq:dfn_sdoper}
	\mathcal F(\omega)
	=(2\pi)^{-1}\sum_{j\in\mathbb Z}\mathcal C(j)e^{-\mathbf ij\omega},
	\quad\omega\in[-\pi,\pi],
\end{equation}
where $\mathbf i=\sqrt{-1}$. $\mathcal F(\omega)$ is well-defined for $\omega\in[-\pi,\pi]$ provided that $\sum_{j\in\mathbb Z}\vert\kern-0.25ex\vert\kern-0.25ex\vert\mathcal C(j)\vert\kern-0.25ex\vert\kern-0.25ex\vert_2<\infty$, where $\vert\kern-0.25ex\vert\kern-0.25ex\vert\cdot\vert\kern-0.25ex\vert\kern-0.25ex\vert_2$ is the Hilbert-Schmidt norm. We say that $\{u_t\}_{t\in \mathbb{Z}}$ is a weak white noise if $\mathcal C(j)=0$ for $j\ne 0$, and evidently in this case $\mathcal F(\omega)=(2\pi)^{-1}\mathcal C(0)$ is constant as a function of $\omega\in[-\pi,\pi]$.

\subsection{Definition of test statistic and null asymptotics}
In order to measure the proximity of a given functional time series process $\{u_t\}_{t\in\mathbb Z}$ to a white noise, it is natural then to consider the distance $Q$, in terms of integrated normed error, between the spectral density operator $\mathcal{F}(\omega)$, $\omega\in[-\pi,\pi]$, and $(2\pi)^{-1}\mathcal C(0)$:
\[
	Q^2
	=2\pi\int_{-\pi}^\pi\opnorm{\mathcal F(\omega)-(2\pi)^{-1}\mathcal C(0)}_2^2d\omega.
\]
$Q$ specifically measures how far the second-order structure of $\{u_t\}_{t\in\mathbb Z}$  deviates from that of a weak white noise. Given a sample $u_1,\ldots,u_n$ from $\{u_t\}_{t\in \mathbb{Z}}$, we are then interested in testing the hypothesis
\[
	H_0:\ Q=0\quad\text{versus}\quad H_1:Q>0.
\]

\cite{zhang:2016} and \cite{bagchi:characiejus:dette:2018} develop general tests of $H_0$ versus $H_1$. In this paper, we focus on testing $H_1$ versus the stronger but still relevant hypothesis
$$
H_{0,\text{iid}}: \mbox{the sequence $\{u_t\}_{t\in\mathbb Z}$ is independent and identically distributed in $\mathbb{H}$}.
$$
We discuss in Section \ref{sec:conc} how one might adapt the proposed test statistic to generally test $H_0$ versus $H_1$, but we do not pursue this in detail here. One can estimate $Q$ via estimates of $\mathcal{F}(\omega)$, $\omega\in[-\pi,\pi]$, and $\mathcal{C}(0)$. The sample autocovariance operators are defined by
\begin{equation}\label{eq:dfn_covest}
	\hat{\mathcal C}_n(j)
	=n^{-1}\sum_{t=j+1}^nu_t\otimes u_{t-j},\quad0\le j<n,
\end{equation}
and $\hat{\mathcal C}_n(j)=\hat{\mathcal C}_n^*(-j)$ for $-n<j<0$, where $^*$ denotes the adjoint of an operator. The spectral density operator may be estimated using a kernel lag-window estimator defined by
\begin{equation}\label{eq:dfn_sdoper_est}
	\hat{\mathcal F}_n(\omega)
	=(2\pi)^{-1}\sum_{|j|<n}k(j/p_n)\hat{\mathcal C}_n(j)e^{-\mathbf ij\omega},\quad\omega\in[-\pi,\pi],
\end{equation}
where $k:\mathbb R\to[-1,1]$ is a kernel and $p_n$ is the lag-window or bandwidth parameter.  We make the following assumptions on the kernel $k$ and $p_n$.
\begin{assumption}\label{assumption:kernel}
$k:\mathbb R\to[-1,1]$ is a symmetric function that is continuous at zero and at all but finite number of points, with $k(0)=1$ and $k(x)=O(x^{-\alpha})$, for some $\alpha > 1/2$ as $x\to\infty$.
\end{assumption}
\begin{assumption}\label{assumption:band}
$p_n$ satisfies that $p_n\to\infty$ and $p_n/n\to0$ as $n\to\infty$.
\end{assumption}
\autoref{assumption:kernel} covers all typically used kernels in the literature on spectral density estimation, and guarantees that integrals of the form $\int_{-\infty}^{\infty} k^r(x)dx$ are finite, for $r=2,4$. The above estimates yield an estimate of $Q$ defined by
\[
	\hat Q_n^2
	=2\pi\int_{-\pi}^\pi\opnorm{\hat{\mathcal F}_n(\omega)-(2\pi)^{-1}\hat{\mathcal C}_n(0)}_2^2d\omega.
\]
Using the fact that $\opnorm{\cdot}_2^2=\langle \cdot,\cdot\rangle_{\mathrm{HS}}$, where $\langle \cdot,\cdot\rangle_{\mathrm{HS}}$ is the Hilbert-Schmidt inner product and the fact that the functions $\{e_k\}_{k\in\mathbb Z}$ defined by $e_k(x)=e^{-{\mathbf i}kx}$ for $k\in\mathbb Z$ and $x\in[-\pi,\pi]$ are orthonormal in $L^2[-\pi,\pi]$, we obtain
\[
	\hat Q_n^2
	=\sum_{0<|j|<n}k^2(j/p_n)\opnorm{\hat{\mathcal C}_n(j)}_2^2.
\]
Further, since the kernel $k$ is symmetric, $\hat{\mathcal C}_n(-j)=\hat{\mathcal C}_n^*(j)$, and $\opnorm{\hat{\mathcal C}_n^*(j)}_2=\opnorm{\hat{\mathcal C}_n(j)}_2$, we also have that
\[
	\hat Q_n^2
	=2\sum_{j=1}^{n-1}k^2(j/p_n)\opnorm{\hat{\mathcal C}_n(j)}_2^2.
\]
Remarkably, $\hat Q_n^2$ can be normalized under these general conditions in order to satisfy the central limit theorem under $H_{0,\text{iid}}$, and we now proceed by defining the normalizing sequences and constants needed to do so. Let $\sigma^2=\operatorname E\|u_0\|^2$, $\mu_4=\operatorname E\|u_0\|^4$ and $\hat\sigma_n^2=n^{-1}\sum_{t=1}^n\|u_t\|^2$. Also, let us denote
\[
	C_n(k)=\sum_{j=1}^{n-1}(1-j/n)k^2(j/p_n)
	\quad\text{and}\quad
	D_n(k)=\sum_{j=1}^{n-2}(1-j/n)(1-(j+1)/n)k^4(j/p_n).
\]
We propose to use the test statistic $T_n$  defined by
\begin{equation}\label{eq:teststat}
	T_n=T_n(k,p_n)
	=\frac{2^{-1}n\hat Q_n^2-\hat\sigma_n^4C_n(k)}{\opnorm{\hat{\mathcal C}_n(0)}_2^2\sqrt{2D_n(k)}},\quad n\ge1.
\end{equation}
We now state our main result, which establishes the asymptotic normality of $T_n$.
\begin{thm}\label{thm:main}
Suppose that $\{u_t\}_{t\in\mathbb Z}$ satisfies $H_{0,\text{iid}}$ and that Assumptions~\ref{assumption:moment}, \ref{assumption:kernel}, and \ref{assumption:band} hold. Then
\begin{equation}\label{eq:mainthm}
	T_n
	\xrightarrow{d}N(0,1)\quad\text{as}\quad n\to\infty.
\end{equation}
\end{thm}
\begin{rmrk}
Let us observe that
\begin{equation}\label{eq:prefactor}
	T_n
	=\frac{\hat\sigma_n^4}{\opnorm{\hat{\mathcal C}_n(0)}_2^2}\cdot\frac{2^{-1}n\hat\sigma_n^{-4}\hat Q_n^2-C_n(k)}{\sqrt{2D_n(k)}},
	\quad n\ge1.
\end{equation}
If $\mathbb H=\mathbb R$, then it follows that $\hat\sigma_n^2/\opnorm{\hat{\mathcal C}_n(0)}_2 \stackrel{p}{\to} 1$ as $n\to \infty$, and we recover precisely the test statistic proposed by \cite{hong:1996} (see the statistic $M_{1n}$ on page 840 of \cite{hong:1996}). In a general real separable Hilbert space $\mathbb H$, $\hat\sigma_n^2/\opnorm{\hat{\mathcal C}_n(0)}_2$ converges in probability to $\sigma^2/\opnorm{\mathcal C(0)}_2$, which need not equal one.  Intuitively, the structure ``inside'' $u_t$ determines the value of $\sigma^2/\opnorm{\mathcal C(0)}_2$.
\end{rmrk}
\begin{rmrk}\label{rem:d-def}
Under Assumptions \ref{assumption:kernel} and \ref{assumption:band},  we have that $D_n(k)= D(k)p_n + o(p_n)$ as $n\to\infty$, where $D(k)=\int_0^\infty k^4(z)dz$.
\end{rmrk}

\begin{rmrk}\label{rem:horv-rem} Taking the kernel $k$ to be the truncated kernel,  $k = \mathds{1}_{\{|x|\le 1\}}$, one has that
\[
	\hat Q_n^2
	=2\sum_{j=1}^{p_n}\opnorm{\hat{\mathcal C}_n(j)}_2^2.
\]
Additionally if $p_n/\sqrt{n} \to 0$, then also  $C_{n}(k)/p_n \to 1$, and $D_n(k)/p_n\to 1$ as $n\to \infty$. It follows then  under these and the conditions of \autoref{thm:main} that the statistic $T_n$ is asymptotically equivalent with
$$
T_n^{*} = \frac{n\sum_{j=1}^{p_n}\opnorm{\hat{\mathcal C}_n(j)}_2^2 - \hat\sigma_n^4p_n}{\opnorm{\hat{\mathcal C}_n(0)}_2^2\sqrt{2p_n}}
$$
which is identical to the statistic considered in \cite{horvath:huskova:rice:2013}. It was shown there that this statistic has a normal limit under $H_{0,\text{iid}}$ and the assumption that $p_n\to \infty$, $p_n=O(\log(n))$ as $n\to \infty$. Therefore \autoref{thm:main} can be viewed as a generalisation of their result.
\end{rmrk}

\subsection{Transformation of test statistic}
\citet{chen2004} investigated the finite sample performance of the test statistic of \citet{hong:1996}, and found that the sampling distribution of the test statistic under $H_{0,iid}$ tends to be right skewed, causing the test to be oversized. We confirm their findings in our simulation study in Section~\ref{simul} as the test statistic $T_n$ suffers from the same problem in the general Hilbert space setting as well. In order to alleviate this, \citet{chen2004} suggests a power transformation of the test statistic of the form
\[
	T_n^\beta=
	\frac{(2^{-1}n\hat\sigma_n^{-4}\hat Q_n^2)^\beta-[C_n^\beta(k)+2^{-1}\beta(\beta-1)C_n^{\beta-2}(k)\hat\sigma_n^{-8}\opnorm{\hat{\mathcal C}_n(0)}_2^42D_n(k)]}
	{\beta C_n^{\beta-1}(k)\hat\sigma_n^{-4}\opnorm{\hat{\mathcal C}_n(0)}_2^2[2D_n(k)]^{1/2}}
\]
for $n\ge1$ with $\beta\ne0$. It follows from \autoref{thm:main} and the application of the delta method that $T_n^\beta\xrightarrow{d}N(0,1)$ as $n\to\infty$ for $\beta\ne0$.

 \citet{chen2004} then recommend to choose the value of $\beta$ in order to make approximate skewness of $T_n^\beta$ equal to zero. In order to describe how to achieve this in our setting, let us suppose that $Z$ is a Gaussian random element with values in $S_2(\mathbb H)$ such that $\operatorname EZ=0$ and the covariance operator of $Z$ is given by $\mathcal A=\operatorname E[u_0\otimes u_1\tilde\otimes u_0\otimes u_1]$, where $S_2(\mathbb H)$ is the space of the Hilbert-Schmidt operators from $\mathbb H$ to $\mathbb H$ and $\tilde\otimes$ denotes the tensor of two elements of $S_2(\mathbb H)$. As a result of the central limit theorem applied to the estimated covariance operators $\hat{\mathcal C}_n(j)$, $2^{-1}n\hat Q_n^2$ is approximately a weighted sum of independent and identically distributed random variables with the distribution of $\opnorm{Z}_2^2$. From this we can obtain the approximate skewness of $(2^{-1}n\hat Q_n^2)^\beta$ using Taylor's theorem (see \citet{chen2004} for more details). The value of $\beta$ which makes the approximate skewness equal to zero is given by
\begin{equation}\label{eq:beta}
	\beta^*
	=1-\frac{\mu_1'\mu_3}{3\mu_2^2}\frac{[\sum_{j=1}^{n-1}k^2(j/p_n)][\sum_{j=1}^{n-1}k^6(j/p_n)]}{[\sum_{j=1}^{n-1}k^4(j/p_n)]^2},
\end{equation}
where $\mu_1'=\operatorname E\opnorm{Z}_2^2$, $\mu_2=\operatorname{Var\opnorm{Z}_2^2}$ and $\mu_3=\operatorname E[\opnorm{Z}_2^2-\operatorname E\opnorm{Z}_2^2]^3$.

If $\mathbb H=\mathbb R$, $\opnorm{Z}_2^2$ is distributed as a scaled $\chi^2(1)$ random variable, and then $\mu_1'\mu_3/(3\mu_2^2)=2/3$. Consequently, $\beta^*$ in the univariate case only depends on the sample size $n$, the kernel $k$ and the bandwidth $p_n$. The choice of $\mu_1'\mu_3/(3\mu_2^2)=2/3$ even in the general setting can be motivated by using a Welch--Satterthwaite style approximation of the norm of a Gaussian process $\opnorm{Z}_2^2$ with a scaled chi-squared distribution; see for instance  \citet{zhang:2013} and \citet{krishnamoorthy:2016}. However, in the general Hilbert space setting the value of $\mu_1'\mu_3/(3\mu_2^2)$ depends on the data and should be estimated. Let us suppose that $\{v_k\}_{k\ge1}$ are the eigenvectors of $\mathcal A$ with the corresponding eigenvalues $\{\lambda_k\}_{k\ge1}$. We have that
\[
	\opnorm{Z}_2^2
	=\sum_{k\ge1}|\langle Z,v_k\rangle_{\mathrm{HS}}|^2
	=\sum_{k\ge1}\lambda_k|\lambda_k^{-1/2}\langle Z,v_k\rangle_{\mathrm{HS}}|^2
	=\sum_{k\ge1}\lambda_k\xi_k,
\]
where $\{\xi_k\}_{k\ge1}$ are independent and identically distributed $\chi^2(1)$ random variables. It follows by calculating the cumulant generating function of $\opnorm{Z}_2^2$ that
\[
	\mu_1'
	=\sum_{k\ge1}\lambda_k
	=\operatorname{Tr}\mathcal A,
	\quad
	\mu_2
	=2\sum_{k\ge1}\lambda_k^2
	=2\operatorname{Tr}\mathcal A^2
	\quad\text{and}\quad
	\mu_3
	=8\sum_{k\ge1}\lambda_k^3
	=8\operatorname{Tr}\mathcal A^3,
\]
where $\operatorname{Tr}$ denotes the trace of an operator and $\mathcal A^k$ denotes the $k$-fold compositon of the operator with itself. By \autoref{lemma:covcov} in Appendix \ref{sec:app:pt},
\[
	\operatorname{Tr}\mathcal A
	=[\operatorname{Tr}\mathcal C(0)]^2,
	\quad
	\operatorname{Tr}\mathcal A^2
	=[\operatorname{Tr}\mathcal C^2(0)]^2
	\quad\text{and}\quad
	\operatorname{Tr}\mathcal A^3
	=[\operatorname{Tr}\mathcal C^3(0)]^2
\]
and we can estimate $\mu_1'$, $\mu_2$ and $\mu_3$ by estimating $\mathcal C(0)$ using the estimator $\hat{\mathcal C}_n(0)$ given by~\eqref{eq:dfn_covest}. It follows that the plug-in estimators of $\mu_1'$, $\mu_2$ and $\mu_3$ are given by
\begin{equation}\label{eq:mu1mu2est}
	\hat\mu_1'
	=n^{-1}\sum_{t=1}^n\|u_t\|^2,
	\quad
	\hat\mu_2
	=n^{-2}\sum_{s,t=1}^n|\langle u_t,u_s\rangle|^2
\end{equation}
and
\begin{equation}\label{eq:mu3est}
	\hat\mu_3
	=n^{-3}\sum_{r,s,t=1}^n\langle u_t,u_s\rangle\langle u_s,u_r\rangle\langle u_r,u_t\rangle.
\end{equation}
These results motivate two tests of $H_{0,\text{iid}}$ each with asymptotic size $\alpha \in (0,1)$: to Reject $H_{0,\text{iid}}$ if $T_n \ge \Phi^{-1}(1-\alpha)$ or $T_n^\beta\ge\Phi^{-1}(1-\alpha)$ with  $\beta\ne0$, where $\Phi^{-1}(1-\alpha)$ is the $1-\alpha$ quantile of the standard normal distribution. The finite sample properties of each of these tests for various choices of the bandwidth and kernel, as well as a comparison of these tests to existing methods, are presented in Section \ref{simul}.

\subsection{Consistency}
We now establish the consistency of our test under $H_1$. In order to describe time series $\{u_t\}_{t\in \mathbb{Z}}$ that are generally weakly dependent, we follow \cite{tavakoli2014} and introduce cumulant summability conditions. Suppose that $\{u_t\}_{t\in\mathbb Z}$ is a sequence of $\mathbb H$-valued random elements. We say that $\{u_t\}_{t\in\mathbb Z}$ is $k$-th order stationary with $k\ge1$ if $\operatorname E\|u_t\|^k<\infty$ for all $t\in\mathbb Z$ and, for all $t_1,\ldots,t_l\in\mathbb Z$ and $l=1,\ldots,k$,
\[
	\operatorname E[u_{t_1}\otimes u_{t_2}\otimes\ldots\otimes u_{t_l}]
	=\operatorname E[u_0\otimes u_{t_2-t_1}\otimes\ldots\otimes u_{t_l-t_1}].
\]
The joint cumulant of real or complex-valued random variables $X_1,\ldots,X_k$ such that $\operatorname E|X_j|^k<\infty$ for all $j=1,\ldots,k$ with $k\ge1$ is given by
\[
	\operatorname{cum}(X_1,\ldots,X_k)
	=\sum_{\pi}(|\pi|-1)!(-1)^{|\pi|-1}\prod_{B\in\pi}\operatorname E\Bigl[\prod_{i\in B}X_i\Bigr],
\]
where $\pi$ runs through the list of all partitions of $\{1,\ldots,n\}$, $B$ runs through the list of all blocks of the partition $\pi$, and $|\pi|$ is the number of parts in the partition; see \cite{brillinger2001} for more details. Suppose that $u_1,\ldots,u_k$ are $\mathbb H$-valued random variables such that $\operatorname E\|u_j\|^k<\infty$ for all $j=1,\ldots,k$. The $k$-th order cumulant $\operatorname{cum}(u_1,\ldots,u_k)$ is a unique element in $\bigotimes_{j=1}^k\mathbb H$ that satisfies
\[
	\langle\operatorname{cum}(u_1,\ldots,u_k),f_1\otimes\ldots\otimes f_k\rangle
	=\operatorname{cum}(\langle u_1,f_1\rangle,\ldots,\langle u_k,f_k\rangle)
\]
for all $f_1,\ldots,f_k\in\mathbb H$. Suppose that $\{u_t\}_{t\in\mathbb Z}$ is $2k$-th order stationary and let us denote
\[
	\mathcal K_{t_1,\ldots,t_{2k-1}}=\operatorname{cum}(u_{t_1},\ldots,u_{t_{2k-1}},u_{t_0})
\]
for $t_1,\ldots,t_{2k-1}\in\mathbb Z$. Then $\mathcal K_{t_1,\ldots,t_{2k-1}}$ is an operator that maps the elements of $\bigotimes_{j=1}^k\mathbb H$ to $\bigotimes_{j=1}^k\mathbb H$ such that
\begin{multline}\label{eq:cumoper}
	\langle\mathcal K_{t_1,\ldots,t_{2k-1}}(f_{k+1}\otimes\ldots\otimes f_{2k}),f_1\otimes\ldots\otimes f_k\rangle\\
	=\operatorname{cum}(\langle u_{t_1},f_1\rangle,\ldots,\langle u_{t_{2k-1}},f_{2k-1}\rangle,\langle u_0,f_{2k}\rangle)
\end{multline}
for all $f_1,\ldots,f_{2k}\in\mathbb H$. The following assumption describes the allowable strength of the serial dependence of the $u_t$'s.
\begin{assumption}\label{assumption:dependentseq}
$\{u_t\}_{t\in\mathbb Z}$ is a fourth order stationary sequence of zero mean random elements with values in a real separable Hilbert space $\mathbb H$ such that $\sum_{j=-\infty}^\infty\opnorm{\mathcal C(j)}_1^2<\infty$ and $\sup_{j\in\mathbb Z}\sum_{h=-\infty}^\infty\opnorm{\mathcal K_{h+j,h,j}}_1<\infty$, where $\opnorm{\cdot}_1$ is the nuclear norm.
\end{assumption}
\autoref{assumption:dependentseq} is essentially a generalisation of the assumption used by \cite[Assumption~A.4]{hong:1996} to a real separable Hilbert space $\mathbb H$. A similar assumption is used by \cite[Assumption~3.1]{zhang:2016}. The next theorem establishes the consistency of our test, which shows that, under the alternative hypothesis, the rate at which $T_n$ diverges to infinity as $n\to\infty$ is $n/p_n^{1/2}$.
\begin{thm}\label{thm:consistency}
Suppose that $\{u_t\}_{t\in\mathbb Z}$ satisfies \autoref{assumption:dependentseq} and \autoref{assumption:kernel}, \ref{assumption:band} hold. Then
\[
	(p_n^{1/2}/n)T_n
	\xrightarrow{p}
	\frac{2^{-1}Q^2}{\opnorm{\mathcal C(0)}_2^2(2D(k))^{1/2}}\quad\text{as}\quad n\to\infty,
\]
where $D(k)$ is defined in \autoref{rem:d-def}.

\end{thm}

\section{Simulation study}\label{simul}

We now present the results of a simulation study which aimed to evaluate the performance of the tests based on $T_n$ and $T_n^\beta$ in finite samples, and to compare the proposed method to some other approaches available in the literature. For this we take the Hilbert space $\mathbb{H}$ to be $L^2([0,1],\mathbb{R})$, i.e.\ the space of equivalence classes of almost everywhere equal square-integrable real valued functions defined on $[0,1]$. In particular, we applied the proposed method to simulated samples from the following data generating processes (DGP's):

\begin{enumerate}
  \item[a)] IID-BM: $u_i(t)=W_{i}(t)$, where $\left\{W_{i}(t),\; u\in[0,1]\right\}_{i=1}^\infty$ is a sequence of independent and identically distributed standard Brownian motions.
  \item[b)] fGARCH(1,1): $u_i(t)$ satisfies

\begin{align}\label{fgarch-form}
 u_i(t) = \sigma_i(t)\varepsilon_i(t), \quad
 \sigma_i^2(t) = \delta(t) + \alpha(u_{i-1}^2)(t) + \beta(\sigma_{i-1}^2)(t),
\end{align}
where $\alpha$ and $\beta$ are integral operators defined, for $x\in L^2([0,1],\mathbb{R})$ and $t\in [0,1]$, by
\begin{equation*}
(\alpha x)(t) = \int 12t(1-t)s(1-s)x(s)ds, \quad
(\beta x)(t) = \int 12t(1-t)s(1-s)x(s)ds,
\end{equation*}
 $\delta = 0.01$ (a constant function), and
\begin{equation*}
\epsilon_i(t) = \frac{\sqrt{\ln(2)}}{2^{200t}}B_i\left(\frac{2^{400t}}{\ln 2}\right), \quad t\in [0,1],
\end{equation*}
where $\left\{B_i(t), t \in [0,1]\right\}_{i\in \mathbb{Z}}$ are independent and identically distributed Brownian bridges.
\item[c)] FAR($1$, $S$)-BM: $u_i(t) = \int \psi_c(t, s) u_{i-1}(s)ds + \eg_i(t)$, $u \in [0,1],$ where $\eg_i$ follows IID-BM, and
$\psi_c(t,s) = c\exp\{ ( t^2 + s^2)/2\}$.
The constant $c$ is then chosen so that $\|\psi_c\|=S\in (0,1)$.
\end{enumerate}

Evidently the data generated from IID-BM satisfies both $H_0$ and $H_{0,\text{iid}}$. The functional GARCH process that we study here was introduced in \cite{aue:horvath:pellatt:2017}, and the particular settings of the operators and error process are meant to imitate high-frequency intraday returns. This process satisfies $H_0$ but not $H_{0,\text{iid}}$. The FAR(1, $S$)-BM satisfies neither $H_0$ nor $H_{0,\text{iid}}$, and we study in particular the case when $S=0.3$ in order to compare to the results in \cite{zhang:2016}. Each random function was generated on 100 equally spaced points on the $[0,1]$ interval, and for the DGP's fGARCH(1,1) and FAR(1, $S$)-BM a burn-in sample of length 30 was generated prior to generating a sample of length $n$.

The tests based on $T_n$ and $T_n^\beta$ require the choice of a kernel function and bandwidth, which we now discuss. The problem of kernel and bandwidth selection for nonparametric spectral density estimation enjoys an enormous literature, going back to the seminal work of \cite{bartlett:1950} and \cite{parzen:1957}. This problem has recently received  attention for general function space valued time series in, for example, \cite{panaretos:tavakoli:2013} and \cite{rice:shang:2017}. Their theoretical and empirical findings each support using standard kernel functions with corresponding bandwidths tuned to the order or ``flatness" of the kernel near the origin. \cite{rice:shang:2017} showed that some further gains can be achieved in terms of estimation error of the spectral density operator at frequency zero by utilizing data driven bandwidths. With this in mind, we considered the following kernel functions
\begin{align*}
k_{\text{B}}(x)&= \left\{ \begin{array}{ll}
         1-|x| & \mbox{for $|x| \leq 1$};\\
        0 & \mbox{otherwise}.\end{array} \right. \tag{Bartlett} \\
k_{\text{P}}(x) &= \left\{ \begin{array}{ll}
 1-6x^2 + 6|x|^3 & \mbox{for $0\leq |x|\leq \frac{1}{2}$}; \\
 2(1-|x|)^3 & \mbox{for $\frac{1}{2}\leq |x|\leq 1$}; \\
 0 & \mbox{otherwise}. \end{array} \right. \tag{Parzen} \\
 k_{\text{D}}(x) &= \frac{\sin(\pi x)}{\pi x}, \;\; x\in \mathbb{R}/ \{0\},\;\;\; k_{\text{D}}(0)=1 . \tag{Daniell}
\end{align*}

The respective orders of these kernels are $1$,$2$, and $2$. For each of these kernels, we considered bandwidths of the form $p_n =  n^{1/(2q+1)} $ where $q$ is the order of the respective kernel, as well as bandwidths of the form $p_n =  \hat{M} n^{1/(2q+1)} $ where $\hat{M}$ is a constant estimated from the data aiming to adapt the size of the bandwidth to the serial dependence of the observed time series. In particular, it aims to minimize the integrated normed error of the spectral density operator based on a pilot estimates of $\mathcal{F}$ and its generalized derivative. The details of this data driven bandwidth selection are given in Appendix~\ref{sec:band} below.

We consider two different choices of $\beta$ for $T_n^\beta$: $\beta^*$ given by \eqref{eq:beta} when $\mu_1'\mu_3/(3\mu_2^2)=2/3$ and when $\mu_1'\mu_3/(3\mu_2^2)$ is estimated from the data using \eqref{eq:mu1mu2est} and \eqref{eq:mu3est}. We denote the former choice of $\beta$ by $\beta_1^*$ and the latter case by $\hat\beta^*$ below.

Our simulations show that the values of $\beta_1^*$ and $\hat\beta^*$ heavily depend on the choice of the kernel. The sample size and the choice of the bandwidth affect the values of $\beta_1^*$ and $\hat\beta^*$ to a much lesser extent. We obtain similar values of $\hat\beta^*$ for the different DGP's that we consider. $\beta_1^*\approx1/5$ when the kernel is either the Bartlett kernel or the Parzen kernel and $\beta_1^*\approx1/7$ when the kernel is the Daniell kernel. $\hat\beta^*\approx-1/8$ when the kernel is the Bartlett kernel, $\hat\beta^*\approx-1/28$ when the kernel is the Parzen kernel and $\hat\beta^*\approx-1/4$ for the Daniell kernel.

In addition to the proposed statistics, we also compared to the statistic/test $Z_n(b)$ of \cite{zhang:2016} and to the statistic/test $BCD_n$ of \citet{bagchi:characiejus:dette:2018}. \citet{zhang:2016} utilizes the choice of a block size $b$ for a block bootstrap. We present the results for  $b=10$ given the relative similarity in performance for different values of $b$.

The number of rejections from 1000 independent simulations with nominal levels of 5\% and 1\% are reported in Table \ref{res_tab}. From this we can draw the following conclusions about the test:

\begin{enumerate}
  \item The tests based on $T_n$ tend to be oversized. Histograms and summaries of the distribution of the test statistic for data generated according to IID-BM indicate that the distribution is right skewed relative to the standard normal distribution, but has approximately mean zero and variance one. These findings are consistent with the findings of \citet{chen2004} and \citet{horvath:huskova:rice:2013}. Applying the power transformation to $T_n$ corrects for this fairly well in finite samples. Our simulation results show that $T_n^{\beta_1^*}$ is well sized, although $T_n^{\hat\beta^*}$ seems to be slightly undersized.
  \item Regarding kernel and bandwidth selection, no one kernel or bandwidth setting displayed substantially superior performance. There were only negligible differences in the results when comparing the data driven bandwidth with standard bandwidths for the DGP's considered. In terms of size all kernels exhibited similar performance,
although the Parzen kernel exhibited slightly higher power relative to other kernels.
  \item The tests based on $T_n$, $T_n^{\beta_1^*}$ or $T_n^{\hat\beta^*}$ are not appropriately sized for the fGARCH(1,1) DGP, which might be expected since these statistics are not adjusted in any way to handle general weak white noise sequences. We provide some further discussion on this in Section \ref{sec:conc}. By contrast the tests of \cite{zhang:2016} and \citet{bagchi:characiejus:dette:2018} are built for such sequences, although the test of \citet{bagchi:characiejus:dette:2018} seems to be undersized.

  \item The tests based on $T_n$, $T_n^{\beta_1^*}$ and $T_n^{\hat\beta^*}$ exhibited higher power when compared to the test $Z_n$ or $BCD_n$ for FAR(1,0.3)-BM data, especially at the level of 1\%.
  Another advantage of the proposed approach over the test of \citet{zhang:2016} is its reduced computational burden. Using a $64$-bit implementation of $\texttt{R}$ running on Windows 10 with an Intel i3-380M (2.53 GHz) processor, one single calculation of the test statistic and/or $p$-value based on $T_n$ or $T_n^{\beta_0}$  when $n=100$ takes less than $1$ second, whereas calculating $Z_n(10)$ with the same data takes several minutes.

  \item In practice, we recommend the use of the statistic $T_n^{\beta_1^*}$ with the data driven bandwidth for general white noise testing as long as conditional heteroscedasticity is not thought to be an issue. In the case when conditional heteroscedasiticy is of concern, the statistic of \cite{zhang:2016} or \citet{bagchi:characiejus:dette:2018} is expected to give more reliable results.
\end{enumerate}

\setlength{\tabcolsep}{4pt}
{\small
\begin{table}[ht]
\centering
\begin{tabular}{rrrrrrrrrrrrrrr}
\multicolumn{1}{c}{DGP:}& \multicolumn{4}{c}{IID-BM} & & \multicolumn{4}{c}{fGARCH(1,1)} & & \multicolumn{4}{c}{FAR(1,0.3)-BM}  \\
 \cmidrule{2-5}\cmidrule{7-10}\cmidrule{12-15}
& \multicolumn{2}{c}{$n=100$} & \multicolumn{2}{c}{$n=250$} && \multicolumn{2}{c}{$n=100$} & \multicolumn{2}{c}{$n=250$} && \multicolumn{2}{c}{$n=100$} & \multicolumn{2}{c}{$n=250$}   \\
Stat/Nominal Size & 5\% & 1\% & 5\% & 1\% && 5\% & 1\% & 5\% & 1\% && 5\% & 1\% & 5\% & 1\% \\
  \hline
$T_n(k_B, n^{1/3} )$ & 50 & 23 & 67 & 34 && 113 & 71 & 142 & 78 && 824 & 749 & 995 & 993 \\
 $T_n(k_B, \hat{M}n^{1/3} )$ & 58 & 24 & 70 & 43 && 110 & 72 & 118 & 73 && 860 & 786 & 998 & 997 \\
 $T_n(k_P, n^{1/5} )$ & 57 & 23 & 63 & 38 && 110 & 72 & 130 & 74 && 868 & 788 & 997 & 996 \\
 $T_n(k_P, \hat{M}n^{1/5} )$& 58 & 22 & 68 & 41 && 111 & 73 & 121 & 71 && 851 & 776 & 997 & 997 \\
  $T_n(k_D, n^{1/5} )$ & 54 & 24 & 63 & 36 && 112 & 69 & 134 & 76 && 860 & 782 & 997 & 995 \\
  $T_n(k_D, \hat{M}n^{1/5} )$ & 58 & 21 & 71 & 29 && 112 & 70 & 127 & 71 && 833 & 761 & 998 & 994 \\
  \hline
  $T_n^{\beta_1^*}(k_B, n^{1/3} )$ & 34 & 6 & 51 & 14 && 91 & 41 & 112 & 35 && 793 & 611 & 995 & 977 \\
  $T_n^{\beta_1^*}(k_B, \hat{M}n^{1/3} )$ & 36 & 8 & 52 & 14 && 86 & 34 & 93 & 25 && 822 & 629 & 997 & 988 \\
  $T_n^{\beta_1^*}(k_P, n^{1/5} )$ & 38 & 10 & 53 & 13 && 87 & 35 & 110 & 31 && 834 & 634 & 997 & 990 \\
 $T_n^{\beta_1^*}(k_P, \hat{M}n^{1/5} )$ & 34 & 7 & 53 & 14 && 86 & 35 & 95 & 25 && 812 & 629 & 997 & 989 \\
  $T_n^{\beta_1^*}(k_D, n^{1/5} )$ & 32 & 6 & 53 & 12 && 80 & 33 & 102 & 28 && 815 & 610 & 997 & 985 \\
  $T_n^{\beta_1^*}(k_D, \hat{M}n^{1/5} )$ & 37 & 4 & 53 & 12 && 82 & 30 & 97 & 23 && 800 & 603 & 997 & 976 \\
  \hline
  $T_n^{\hat\beta^*}(k_B, n^{1/3} )$ & 26 & 2 & 46 & 7 && 80 & 24 & 94 & 17 && 777 & 476 & 993 & 954 \\
  $T_n^{\hat\beta^*}(k_B, \hat{M}n^{1/3} )$ & 31 & 2 & 49 & 4 && 80 & 21 & 86 & 11 && 801 & 482 & 997 & 959 \\
  $T_n^{\hat\beta^*}(k_P, n^{1/5} )$ & 36 & 2 & 44 & 4 && 78 & 22 & 91 & 16 && 813 & 485 & 997 & 960 \\
  $T_n^{\hat\beta^*}(k_P, \hat{M}n^{1/5} )$ & 26 & 2 & 48 & 4 && 80 & 22 & 86 & 10 && 796 & 482 & 997 & 959 \\
  $T_n^{\hat\beta^*}(k_D, n^{1/5} )$ & 26 & 2 & 40 & 4 && 72 & 18 & 85 & 13 && 785 & 416 & 996 & 948 \\
  $T_n^{\hat\beta^*}(k_D, \hat{M}n^{1/5} )$ & 25 & 2 & 44 & 5 && 73 & 18 & 79 & 8 && 773 & 436 & 996 & 946 \\
  \hline
  $Z_n(10)$ & 48 & 9 & 49 & 11 && 50 & 12 & 41 & 5 && 708 & 386 & 992 & 913 \\
  $BCD_n$ & 9 & 0 & 25 & 5 && 24 & 4 & 37 & 4 && 124 & 43 & 376 & 174 \\
   \hline
\end{tabular}\caption{Rejections from 1000 independent simulations with nominal levels of $5\%$ and $1\%$ for each method and DGP considered.}\label{res_tab}
\end{table}
}

\section{Empirical example}\label{sec:emp}
In order to illustrate the utility of the proposed tests, we present here the results of an application to daily Eurodollar futures curves. A Eurodollar futures contract represents an obligation to deliver 1,000,000 USD to a bank outside of the United States at a specified time, and their prices are given as values between zero and 100 defining the interest rate on the transaction. The specific data that we consider are daily settlement prices available at monthly delivery dates for the first six months, and quarterly delivery dates for up to 10 years into the future. Following \cite{kargin:onatski:2008}, we transformed this raw data into smooth curves using cubic $B$-splines, and these curves were reevaluated at 30 day ``monthly" intervals to produce the discretely observed curves that we used in subsequent analyses. The corresponding daily Eurodollar futures curves from the year 1994 are illustrated in the left hand panel of Figure \ref{fig:curves}.

\begin{figure}
        \centering
 \mbox{\subfigure{\includegraphics[width=3.3in]{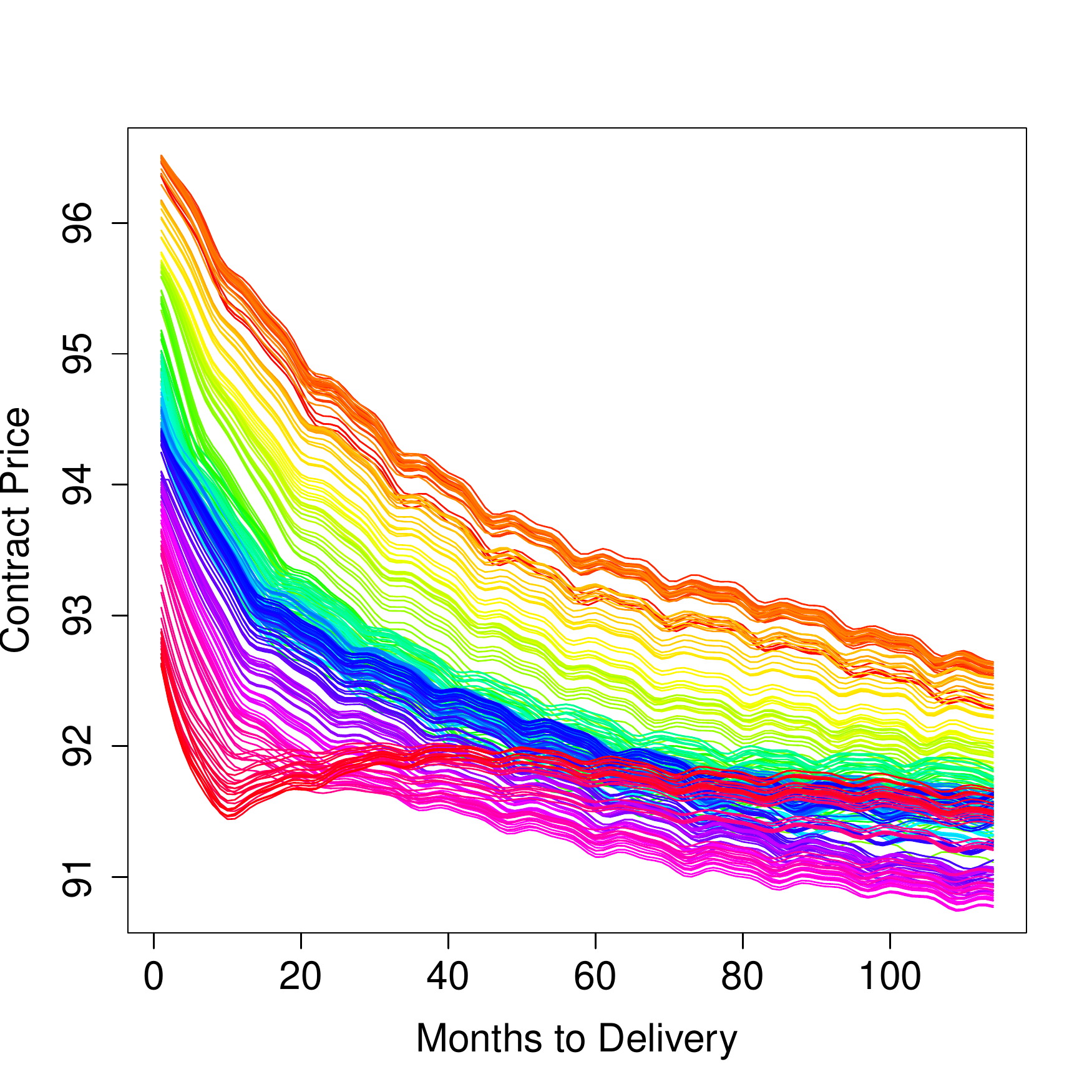}} \subfigure{\includegraphics[width=3.3in]{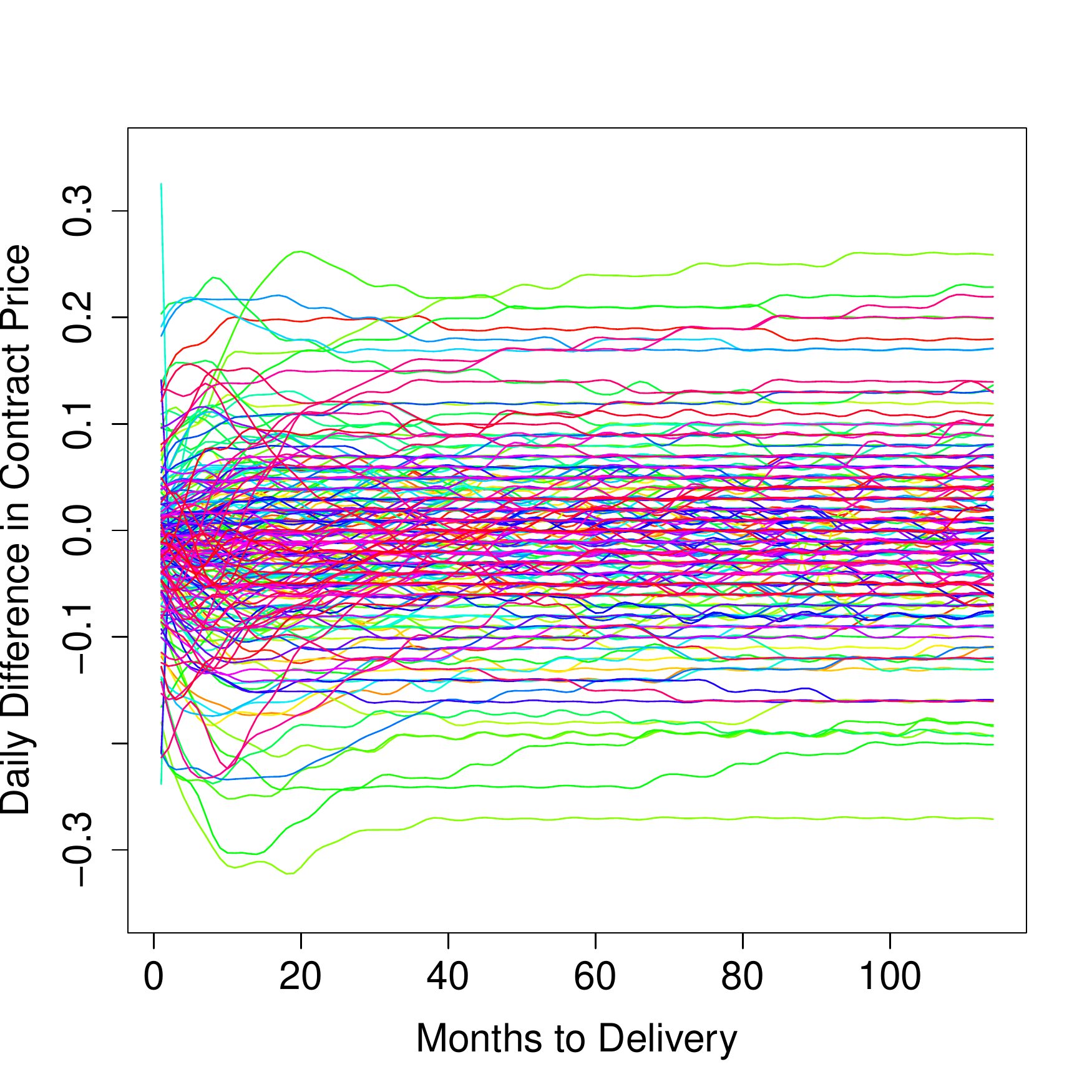} }   } \\
 \vspace{-.5cm}\caption{ Left hand panel: A rainbow plot of the daily Eurodollar futures curves from 1994 obtained from cubic spline smoothing. Right hand panel: A rainbow plot of the first differenced Eurodollar futures curves. Rainbow plots show earlier curves in red and progresses through the visible spectrum showing later curves in violet.   }\label{fig:curves}
\end{figure}

We considered data spanning 10 years from 1994 to 2003, consisting of approximately 2,500 curves. We treated these data as 10 yearly samples of functional time series each of length approximately 250. The basic question we wish to address in each sample is whether or not the curves in that year seem to exhibit significant serial dependence as measured by their autocovariance operators. We applied the proposed test based on the power-transformed statistic $T_n^{\beta_0}$ using the Bartlett kernel and corresponding empirical bandwidth of the form $\hat{M}n^{1/3}$ to each sample. The approximate $p$-values of these tests are displayed in the right hand panel of Figure \ref{fig:pvals}, which are essentially equal to zero in all cases. This suggests that the Eurodollar futures curves exhibit substantial serial dependence. This observation is consistent with the suggested FAR(1) model for these curves proposed by \cite{kargin:onatski:2008}.

Although these results are consistent with the data following an FAR(1) model, they may also be explained by the fact that the raw Eurodollar futures time series are apparently not mean stationary; over periods as long as a year they typically exhibit strong trends and seasonality. We evaluated the stationarity of each of these samples using the test proposed in \cite{horvath:kokoszka:rice:2014}, which suggest that in general the raw Eurodollar futures curves are non-stationary. Letting $X_i(t)$ denote the futures curve on day $i$, we studied then instead the first order differenced curves $Y_i(t)=X_i(t)-X_{i-1}(t)$. The first order differenced Eurodollar futures curves from 1994 are shown in the right hand panel of Figure \ref{fig:curves}, and the stationarity test of \cite{horvath:kokoszka:rice:2014} applied to these curves suggest that they are reasonably stationary. The results of these stationarity tests are illustrated in the left hand panel of Figure \ref{fig:pvals}.

We applied the proposed test using the statistic $T_n^{\beta_0}$ with the same settings as above to each sample of first differenced curves. In six of the ten years considered the hypothesis that the first differenced futures curves evolve as a functional white noise cannot be rejected at the $0.05$ level. Interestingly however, in consecutive years from 1998 to 2001 the first differenced Eurodollar futures curves exhibit significant autocovariance operators to the $0.05$ level as measured by our tests.

\begin{figure}
        \centering
 \mbox{\subfigure{\includegraphics[width=3.1in]{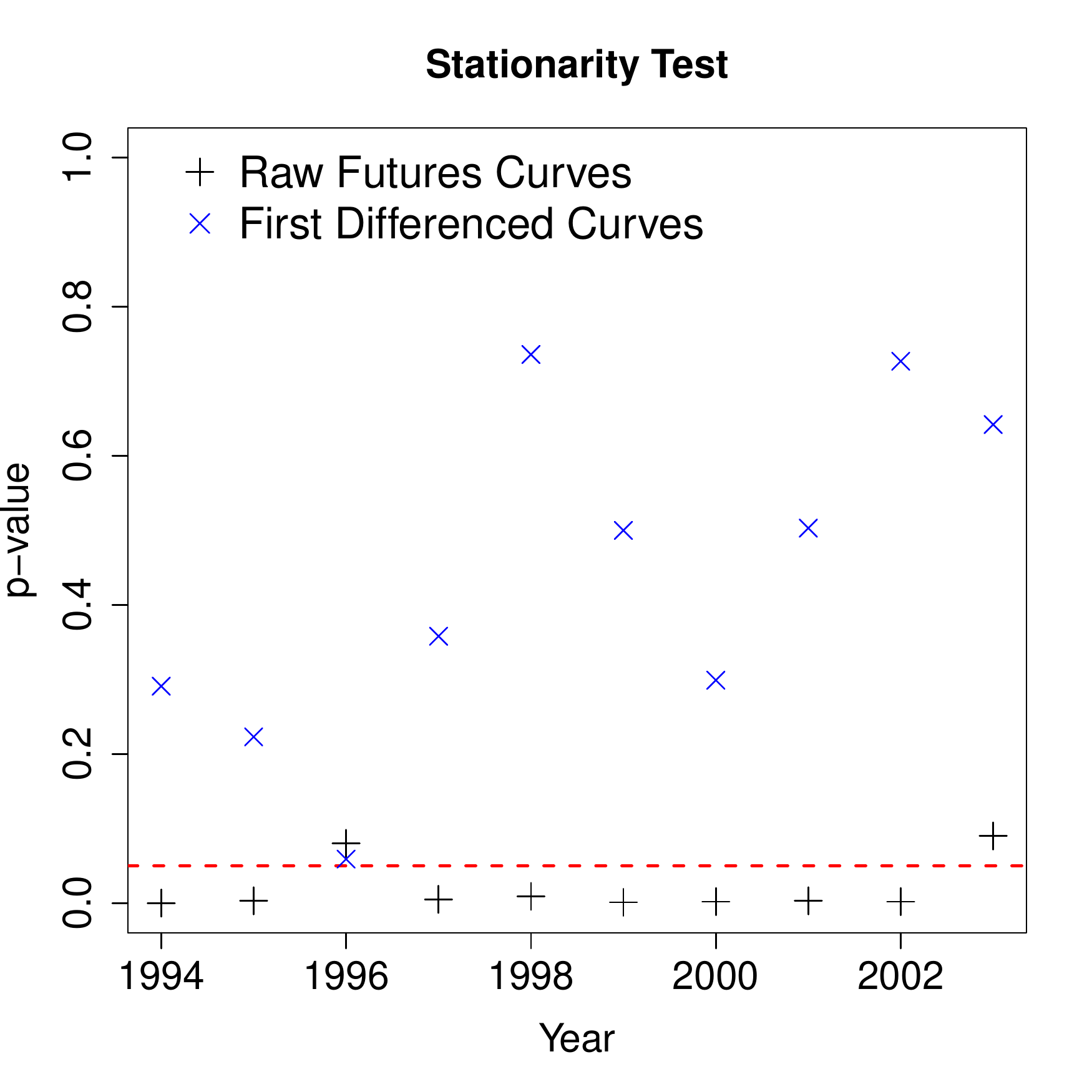}} \subfigure{\includegraphics[width=3.1in]{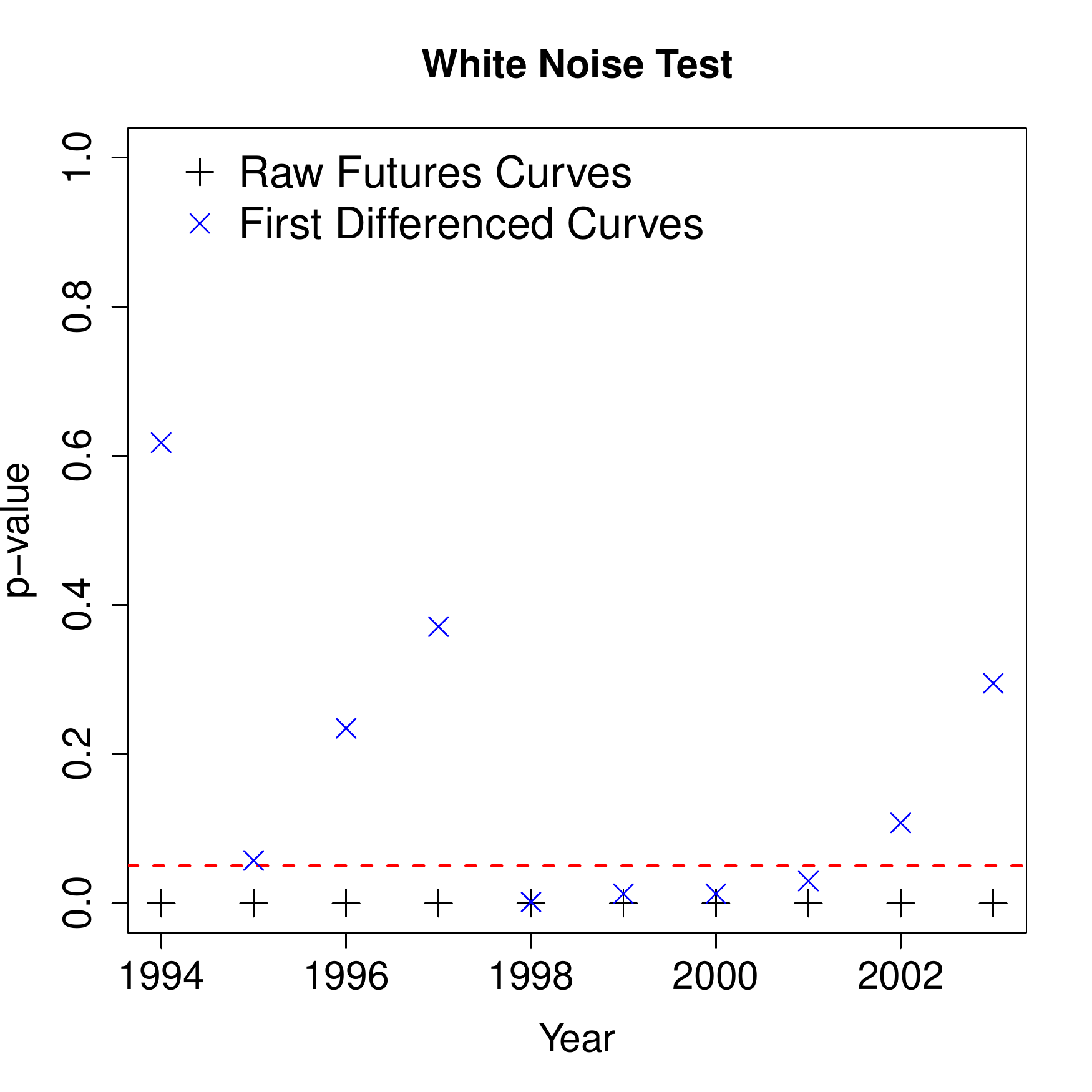} }   } \\
 \vspace{-.5cm}\caption{ Left hand panel: p-values of the stationarity test proposed in \cite{horvath:kokoszka:rice:2014} against the corresponding year of the sample. Right hand panel: p-values of the proposed white noise test against the corresponding years of the sample. In each figure $+$ corresponds to the raw Eurodollar futures curves, and \textcolor[rgb]{0.00,0.07,1.00}{$\times$} corresponds to the first differenced curves. }\label{fig:pvals}
\end{figure}

\section{Conclusion}\label{sec:conc}

We have introduced a new test statistic for white noise testing with functional time series based on kernel lag-window estimates of the spectral density operator. The asymptotic properties of the proposed test have been established assuming the observed time series is a strong white noise, and it was also shown to be consistent for general time series exhibiting serial correlation. The test seems to improve upon existing tests in terms of power against functional autoregressive alternatives, although it has the drawback that it is not well sized for general weak white noise sequences in function space, such as for functional GARCH processes.

Based on the work of \cite{shao:2011} and \cite{shao:2011:2}, we conjecture that \autoref{thm:main} can be established for general weakly dependent white noise sequences in $\mathbb{H}$, despite the fact the proposed tests have markedly inflated size for sequences exhibiting conditional heteroscedasticity in our simulation. Roughly speaking, \cite{shao:2011} shows that in the scalar case the limit distribution of $T_n$ is determined asymptotically by the normalized autocovariances of the process at long lags, which behave essentially the same with strong white noise sequences as with weakly dependent white noises. With this in mind and in light of our simulation results, one expects to need very long time series and a large bandwidth in order for this asymptotic result to be predictive of the behavior of the test statistic for general weak white noise sequence exhibiting serial dependence, which encourages the use of a block bootstrap in finite samples. We leave these issues as potential directions for future research.
\section*{Acknowledgements}
The first author would like to acknowledge the support of the Communaut\'e Fran\c caise de Belgique, Actions de Recherche Concert\' ees, Projects Consolidation 2016--2021. The second author is partially supported by the Natural Science and Engineering Research Council of Canada's Discovery and Accelerator grants. We would like to thank Professor Alessandra Luati for directing us to the work of \citet{chen2004} and we would also like to thank the anonymous reviewers as well as the associate editor for their comments and suggestions that helped us improve this work substantially.


\bibliographystyle{apalike}

\bibliography{garch}

\appendix

\section*{Appendix}\label{sec:app}
\section{Proof of \autoref{thm:main}}

{\small
We begin with a note comparing the basic approach here to that of \cite{horvath:huskova:rice:2013}. \cite{horvath:huskova:rice:2013} uses at its core a central limit theorem for vectors with increasing dimension adapted from \cite{senatov:1998}, and as a result the fairly restrictive condition on the bandwidth $p_n = O(\log(n))$ is needed. Our proof is somewhat more straightforward in the sense that we show that the suitably normalized statistic $\hat{Q}_n^2$ is a martingale with uniformly asymptotically negligible increments. This requires a number of intermediary approximations. Before we prove \autoref{thm:main}, we state two elementary lemmas that we use in the proof.
\begin{lemma}\label{lemma:expofinnerprod}
Suppose that $X$ and $Y$ are independent random elements with values in a separable Hilbert space $\mathbb H$ such that $\operatorname E\|X\|<\infty$, $\operatorname E\|Y\|<\infty$ and  $\operatorname EX=0$ or $\operatorname EY=0$. Then $\operatorname E\langle X,Y\rangle=0$.
\end{lemma}
\begin{lemma}\label{lemma:expofinnerprodsq}
Suppose that $X$ and $Y$ are independent and identically distributed random elements with values in a separable Hilbert space $\mathbb H$ with zero means and finite second moments. Then
\[
	\operatorname E|\langle X,Y\rangle|^2
	=\opnorm{\operatorname E[X\otimes X]}_2^2.
\]
\end{lemma}
Now we are ready to prove \autoref{thm:main}. The basic format of the proof follows \cite{hong:1996}, but in some instances must significantly deviate due to the assumption that the underlying variables are in an arbitrary separable Hilbert space. We denote $k_{nj}=k(j/p_n)$ and $Z_{jt}=u_t\otimes u_{t-j}$.
\begin{proof}[Proof of  \autoref{thm:main}]
By \eqref{eq:prefactor} and Slutsky's theorem, it suffices to show that $\hat\sigma_n^2/\opnorm{\hat{\mathcal C}_n(0)}_2\xrightarrow{p}\sigma^2/\opnorm{\mathcal C(0)}_2$ as $n\to\infty$ and
\begin{equation}\label{eq:nonunitvariance}
	\frac{2^{-1}n\hat\sigma_n^{-4}\hat Q_n^2-C_n(k)}{\sqrt{2D_n(k)}}\xrightarrow{d}N(0,\opnorm{\mathcal C(0)}_2^4/\sigma^8)\quad\text{as}\quad n\to\infty.
\end{equation}
By the law of large numbers, we have that $\hat\sigma_n^2\xrightarrow{p}\sigma^2$ as $n\to\infty$. By \autoref{lemma:cov_con} below, $\opnorm{\hat{\mathcal C}_n(0)-\mathcal C(0)}_2\xrightarrow{p}0$ as $n\to\infty$ and, using the reverse triangle inequality, $\opnorm{\hat{\mathcal C}_n(0)}_2\xrightarrow{p}\opnorm{\mathcal C(0)}_2$ as $n\to\infty$. Hence, $\hat\sigma_n^2/\opnorm{\hat{\mathcal C}_n(0)}_2\xrightarrow{p}\sigma^2/\opnorm{\mathcal C(0)}_2$ as $n\to\infty$.

Now we show \eqref{eq:nonunitvariance} holds. By Chebyshev's inequality, $\hat\sigma_n^2-\sigma^2=O_p(n^{-1/2})$ as $n\to\infty$. Since
\begin{equation}\label{eq:exphatQ_n^2}
	\operatorname E\hat Q_n^2
	=2\sum_{j=1}^{n-1}k_{jn}^2\operatorname E\opnorm{\hat{\mathcal C}_n(j)}_2^2
	=2\sigma^4n^{-1}\sum_{j=1}^{n-1}(1-j/n)k_{jn}^2
	=2\sigma^4n^{-1}C_n(k),
\end{equation}
Markov's inequality implies that $\hat Q_n^2=O_p(p_n/n)$ as $n\to\infty$. We obtain
\[
	\hat\sigma_n^{-4}\hat Q_n^2
	=\sigma^{-4}\hat Q_n^2+(\hat\sigma_n^{-4}-\sigma^{-4})\hat Q_n^2
	=\sigma^{-4}\hat Q_n^2+o_p(p_n^{1/2}/n)
	\quad\text{as}\quad n\to\infty
\]
because $p_n/n\to0$ as $n\to\infty$. Hence,
\[
	\frac{2^{-1}n\hat\sigma_n^{-4}\hat Q_n^2-C_n(k)}{\sqrt{2 D_n(k)}}
	=\frac{2^{-1}n\sigma^{-4}\hat Q_n^2-C_n(k)}{\sqrt{2 D_n(k)}}+o_p(1)
	\quad\text{as}\quad n\to\infty.
\]
Since $\hat{\mathcal C}_n(j)=n^{-1}\sum_{t=j+1}^nZ_{jt}$,
\[
	\opnorm{\hat{\mathcal C}_n(j)}_2^2
	=n^{-2}\sum_{t,s=j+1}^n\langle Z_{jt},Z_{js}\rangle_{\mathrm{HS}}
	=n^{-2}\Big[\sum_{t=j+1}^n\opnorm{Z_{jt}}_2^2+\sum_{t=j+2}^n\sum_{s=j+1}^{t-1}2\langle Z_{jt},Z_{js}\rangle_{\mathrm{HS}}\Bigr].
\]
We obtain
\[
	2^{-1}n\sigma^{-4}\hat Q_n^2
	=n\sigma^{-4}\sum_{j=1}^{n-1}k_{jn}^2\opnorm{\hat{\mathcal C}_n(j)}_2^2
	=\sigma^{-4}(\tilde C_n+\tilde W_n),
\]
where
\begin{equation}\label{eq:doubleprod}
	\tilde C_n
	=n^{-1}\sum_{j=1}^{n-1}\sum_{t=j+1}^nk_{nj}^2\opnorm{Z_{jt}}_2^2
	\quad\text{and}\quad
	\tilde W_n
	=n^{-1}\sum_{j=1}^{n-2}\sum_{t=j+2}^n\sum_{s=j+1}^{t-1}2k_{nj}^2\langle Z_{jt},Z_{js}\rangle_{\mathrm{HS}}.
\end{equation}

Let us observe that
\[
	\operatorname E\tilde C_n
	=\sigma^4\sum_{j=1}^{n-1}(1-j/n)k_{nj}^2
	=\sigma^4C_n(k),
\]
\[
	\operatorname E|\sigma^{-4}\tilde C_n-C_n(k)|^2
	=\operatorname E|\sigma^{-4}\tilde C_n-\sigma^{-4}\operatorname E\tilde C_n|^2
	=\sigma^{-8}\operatorname{Var}\tilde C_n
\]
and
\[
	\operatorname{Var}\tilde C_n
	=\operatorname E|\tilde C_n-\operatorname E\tilde C_n|^2
	=n^{-2}\operatorname E\Bigl|\sum_{j=1}^{n-1}k_{nj}^2\sum_{t=j+1}^n(\opnorm{Z_{jt}}_2^2-\sigma^4)\Bigr|^2.
\]
Since $\operatorname E|\sum_{t=j+1}^n(\opnorm{Z_{jt}}_2^2-\sigma^4)|^2=O(n)$ as $n\to\infty$ for each $0<j<n$, by Minkowski's inequality,
\begin{equation}\label{eq:cor-eq}
	\operatorname{Var}\tilde C_n
	\le\biggl|n^{-1}\sum_{j=1}^{n-1}k_{nj}^2\Bigl(\operatorname E\Bigl|\sum_{t=j+1}^n(\opnorm{Z_{jt}}_2^2-\sigma^4)\Bigr|^2\Bigr)^{1/2}\biggr|^2
	=O(p_n^2/n)
	\quad\text{as}\quad n\to\infty.
\end{equation}
Since $p_n/n\to0$ as $n\to\infty$, we have that $p_n^{-1/2}(\sigma^{-4}\tilde C_n-C_n(k))=o_p(1)$ as $n\to\infty$ and
\[
	\frac{2^{-1}n\sigma^{-4}\hat Q_n^2-C_n(k)}{\sqrt{2 D_n(k)}}
	=\frac{\sigma^{-4}(\tilde C_n+\tilde W_n)-C_n(k)}{\sqrt{2 D_n(k)}}
	=\frac{\tilde W_n}{\sqrt{2\sigma^8D_n(k)}}
	+o_p(1)
	\quad\text{as}\quad n\to\infty.
\]

So we must show that $\tilde W_n/\sqrt{2\sigma^8D_n(k)}\xrightarrow{d}N(0,\opnorm{\mathcal C(0)}_2^4/\sigma^8)$ as $n\to\infty$. Let us choose a real sequence $\{l_n\}_{n\ge1}$ such that $p_n/l_n + l_n/n\to0$ as $n\to\infty$. With $w_{jts}=2\langle Z_{jt},Z_{js}\rangle_{\mathrm{HS}}$, we have that
\[
	\tilde W_n
	=\Bigl[n^{-1}\sum_{j=1}^{l_n}k_{nj}^2+n^{-1}\sum_{j=l_n+1}^{n-2}k_{nj}^2\Bigr]\sum_{t=j+2}^n\sum_{s=j+1}^{t-1}w_{jts}
	=W_{1n}+V_{1n}.
\]
Next, we partition
\begin{align*}
	W_{1n}
	&=n^{-1}\sum_{j=1}^{l_n}k_{nj}^2\Big[\sum_{t=2l_n+3}^n\sum_{s=l_n+2}^{t-l_n-1}+\sum_{t=2l_n+3}^n\sum_{s=t-l_n}^{t-1}+\sum_{t=l_n+3}^{2l_n +2}\sum_{s=l_n+2}^{t-1}+\sum_{s=j+1}^{l_n+1}\sum_{t=s+1}^n\Bigr]w_{jts}\\
	&=U_n+V_{2n}+V_{3n}+V_{4n}.
\end{align*}
Hence, $\tilde W_n=U_n+\sum_{j=1}^4V_{nj}$. One may show that $p_n^{-1/2}V_{jn}=o_p(1)$ as $n\to\infty$ for $j=1,\ldots,4$. We provide a complete argument to show $p_n^{-1/2}V_{1n}=o_p(1)$, and omit the details for $V_{jn}$, $j=2,3$ and 4 since they are similar. We have that
\begin{align*}
	|V_{1n}|^2
	&=n^{-2}\sum_{j=l_n+1}^{n-2}k_{nj}^4\biggl|\sum_{t=j+2}^n\sum_{s=j+1}^{t-1}w_{jts}\biggr|^2\\
	&\quad+2n^{-2}\sum_{j_1=l_n+2}^{n-2}\sum_{j_2=l_n+1}^{n+2}\biggl(k_{nj_1}^2\sum_{t_1=j_1+2}^n\sum_{s_1=j_1+1}^{t_1-1}w_{j_1t_1s_1}\biggr)\biggl(k_{nj_2}^2\sum_{t_2=j_2+2}^n\sum_{s_2=j_2+1}^{t_2-1}w_{j_2t_2s_2}\biggr)\\
	&=A_{1n}+B_{1n}.
\end{align*}
Let us consider the expected value $\operatorname E[w_{j_1t_1s_1}w_{j_2t_2s_2}]$. If $j_1=j_2=j$, the expected value is equal to $0$ by \autoref{lemma:expectation0} unless $s_1=s_2$ and $t_1=t_2$. Hence, using the fact that $\operatorname E|w_{jts}|^2\le 4\mu_4^2$,
\[
	\operatorname E\biggl|\sum_{t=j+2}^n\sum_{s=j+1}^{t-1}w_{jts}\biggr|^2
	=\sum_{t=j+2}^n\sum_{s=j+1}^{t-1}\operatorname E|w_{jts}|^2
	=O(n^2)
\]
and $\operatorname EA_{1n}=o(p_n)$ since $l_n\to\infty$ as $n\to\infty$. Also,
\begin{align*}
	\operatorname EB_{1n}
	&=2n^{-2}\sum_{j_1=l_n+2}^{n-2}k_{nj_1}^2\sum_{j_2=l_n+1}^{j_1-1}k_{nj_2}^2\bigg(\sum_{t_1=j_1+2}^n\sum_{s_1=j_1+1}^{t_1-1}\biggr)\bigg(\sum_{t_2=j_2+2}^n\sum_{s_2=j_2+1}^{t_2-1}\operatorname E[w_{j_1t_1s_1}w_{j_2t_2s_2}]\biggr)\\
	&=O(p_n^2/n)
\end{align*}
as $n\to\infty$ since $\operatorname E[w_{j_1t_1s_1}w_{j_2t_1s_2}]=0$ by \autoref{lemma:expectation0} unless $t_1=t_2$, $s_1=t_1-j_2$ and $s_2=t_1-j_1$ using the fact that $j_1>j_2$. Hence, $\operatorname E|V_{1n}|^2=o(p_n)$ as $n\to\infty$.

Now for $U_n$, we have further that
\begin{align*}
	U_n
	&=n^{-1}\sum_{j=1}^{l_n}k_{nj}^2\sum_{t=2l_n+3}^n\sum_{s=l_n+2}^{t-l_n-1}2\langle Z_{jt},Z_{js}\rangle_{\mathrm{HS}}\\
	&=n^{-1}\sum_{t=2l_n+3}^n2\sum_{j=1}^{l_n}k_{nj}^2\Bigl\langle Z_{jt},\sum_{s=l_n+2}^{t-l_n-1}Z_{js}\Bigr\rangle_{\mathrm{HS}}\\
	&=n^{-1}\sum_{t=2l_n+3}^nU_{nt},
\end{align*}
where
\[
	U_{nt}
	=2\sum_{j=1}^{l_n}k_{nj}^2\langle Z_{jt},H_{jt-l_n-1}\rangle_{\mathrm{HS}}
	\quad\text{and}\quad
	H_{jt-l_n-1}
	=\sum_{s=l_n+2}^{t-l_n-1}Z_{js}.
\]
To establish the asymptotic distribution of $T_n$, we show that $U_n/\sqrt{2\sigma^8D_n(k)}\xrightarrow{d}N(0,\opnorm{\mathcal C(0)}_2^4/\sigma^8)$ as $n\to\infty$. We have that
\begin{align*}
	\operatorname E[U_{nt}\mid F_{t-1}]
	&=2\sum_{j=1}^{l_n}k_{nj}^2\sum_{s=l_n+2}^{t-l_n-1}\operatorname E[\langle u_t,u_s\rangle\langle u_{t-j},u_{s-j}\rangle\mid F_{t-1}]\\
	&=2\sum_{j=1}^{l_n}k_{nj}^2\sum_{s=l_n+2}^{t-l_n-1}\langle u_{t-j},u_{s-j}\rangle\operatorname E[\langle u_t,u_s\rangle\mid F_{t-1}]
\end{align*}
using the fact that $\langle u_{t-j},u_{s-j}\rangle$ is $F_{t-1}$-measurable. Suppose that $\{e_n\}_{n\ge1}$ is an orthonormal basis of $\mathbb H$. Then
\begin{equation}\label{eq:interchange}
	\operatorname E[\langle u_t,u_s\rangle\mid F_{t-1}]
	=\sum_{i=1}^\infty\operatorname E[\langle u_t,e_i\rangle\langle e_i,u_s\rangle\mid F_{t-1}]
	=\sum_{i=1}^\infty\langle e_i,u_s\rangle\operatorname E\langle u_t,e_i\rangle
	=0
\end{equation}
using the fact that $\langle e_i,u_s\rangle$ is $F_{t-1}$-measurable. Using the Cauchy-Schwarz inequality, $H_{0,\text{iid}}$, and Parseval's identity,
\[
	\sum_{i=1}^\infty\operatorname E|\langle u_t,e_i\rangle\langle e_i,u_s\rangle|
	\le\sum_{i=1}^\infty(\operatorname E|\langle u_t,e_i\rangle|^2)^{1/2}(\operatorname E|\langle e_i,u_s\rangle|^2)^{1/2}
	=\operatorname E\|u_0\|^2<\infty.
\]
This shows that the interchange of the expected value and the series in \eqref{eq:interchange} is justified. Hence, $\{U_{nt},F_{t-1}\}$ is a martingale difference sequence, where $F_t$ is the $\sigma$-field generated by $u_s$, $s\le t$.

Let us denote $\sigma^2(n)=\operatorname EU_n^2$. Theorem~2 of \cite{brown:1971} states that $\sigma^{-1}(n)U_n\xrightarrow{d}N(0,1)$ if both of the following conditions hold:
\begin{enumerate}[(a)]
\item\label{a}$\sigma^{-2}(n)n^{-2}\sum_{t=2l_n+3}^n\operatorname E[U_{nt}^2I_{\{|U_{nt}|>\varepsilon n\sigma(n)\}}]\to0$ as $n\to\infty$ for each $\varepsilon>0$;
\item\label{b}$\sigma^{-2}(n)n^{-2}\sum_{t=2l_n+3}^n\ddot U_{nt}^2\xrightarrow{p}1$ as $n\to\infty$, where $\ddot U_{nt}^2=\operatorname E[U_{nt}^2\mid F_{t-1}]$.
\end{enumerate}

We show that $\sigma^2(n)/(2\sigma^8p_nD(k))\to\opnorm{\mathcal C(0)}_2^4/\sigma^8$ as $n\to\infty$. We have that
\[
	\sigma^2(n)
	=n^{-2}\operatorname E\Bigl|\sum_{t=2l_n+3}^nU_{nt}\Bigr|^2
	=n^{-2}\Bigl[\sum_{t=2l_n+3}^n\operatorname EU_{nt}^2+2\sum_{s=2l_n+4}^n\sum_{t=2l_n+3}^{s-1}\operatorname E[U_{ns}U_{nt}]\Bigr].
\]
Let us observe that $Z_{it}=u_t\otimes u_{t-i}$  is independent of $H_{jt-l_n-1}=\sum_{s=l_n+2}^{t-l_n-1}u_s\otimes u_{s-j}$ for $2l_n+3\le t\le n$ and $1\le i,j\le l_n$. We obtain
\begin{multline*}
	\operatorname EU_{nt}^2
	=4\sum_{j=1}^{l_n}k_{nj}^4\operatorname E|\langle Z_{jt},H_{jt-l_n-1}\rangle_{\mathrm{HS}}|^2\\
	+8\sum_{j=2}^{l_n}\sum_{i=1}^{j-1}k_{nj}^2k_{ni}^2\operatorname E[\langle Z_{jt},H_{jt-l_n-1}\rangle_{\mathrm{HS}}\langle Z_{it},H_{it-l_n-1}\rangle_{\mathrm{HS}}].
\end{multline*}
We have that
\[
	\operatorname E|\langle Z_{jt},H_{jt-l_n-1}\rangle_{\mathrm{HS}}|^2
	=\sum_{s=l_n+2}^{t-l_n-1}\operatorname E|\langle Z_{jt},Z_{js}\rangle_{\mathrm{HS}}|^2
	+2\sum_{s=l_n+3}^{t-l_n-1}\sum_{r=l_n+2}^{s-1}\operatorname E[\langle Z_{jt},Z_{js}\rangle_{\mathrm{HS}}\langle Z_{jt},Z_{jr}\rangle_{\mathrm{HS}}].
\]
Let us notice that $u_t$, $u_{t-j}$, $u_s$ and $u_{s-j}$ are independent random elements for $2l_n+3\le t\le n$, $l_n+2\le s\le t-l_n-1$ and $1\le j\le l_n$. Hence, $\langle u_t,u_s\rangle$ and $\langle u_{s-j},u_{t-j}\rangle$ are independent as well and
\begin{align*}
	\operatorname E|\langle Z_{jt},Z_{js}\rangle_{\mathrm{HS}}|^2
	=\operatorname E|\langle u_t,u_s\rangle\langle u_{s-j},u_{t-j}\rangle|^2
	=\operatorname E|\langle u_t,u_s\rangle|^2\operatorname E|\langle u_{s-j},u_{t-j}\rangle|^2
	=\opnorm{\mathcal C(0)}_2^4
\end{align*}
by \autoref{lemma:expofinnerprodsq}. Since $r-j<s-j$,
\[
	\operatorname E[\langle Z_{jt},Z_{js}\rangle_{\mathrm{HS}}\langle Z_{jt},Z_{jr}\rangle_{\mathrm{HS}}]
	=\operatorname E\langle u_{r-j},\langle u_s,u_t\rangle\langle u_{t-j},u_{s-j}\rangle\langle u_r,u_t\rangle u_{t-j}\rangle
	=0
\]
using \autoref{lemma:expofinnerprod}. Since $i<j$,
\begin{multline*}
	\operatorname E[\langle Z_{jt},H_{jt-l_n-1}\rangle_{\mathrm{HS}}\langle Z_{it},H_{it-l_n-1}\rangle_{\mathrm{HS}}]=\\
	=\sum_{s=l_n+2}^{t-l_n-1}\sum_{r=l_n+2}^{t-l_n-1}\operatorname E\langle\langle u_t,u_s\rangle\langle u_{s-j},u_{t-j}\rangle\langle u_t,u_r\rangle u_{r-i},u_{t-i}\rangle
	=0.
\end{multline*}

Finally,
\begin{align*}
	\operatorname E[U_{ns}U_{nt}]
	&=4\operatorname E\bigg[\Big[\sum_{i=1}^{l_n}k_{ni}^2\langle Z_{is},H_{is-l_n-1}\rangle_{\mathrm{HS}}\Bigr]\Bigl[\sum_{j=1}^{l_n}k_{nj}^2\langle Z_{jt},H_{jt-l_n-1}\rangle_{\mathrm{HS}}\Bigr]\biggr]\\
	&=4\sum_{i=1}^{l_n}\sum_{j=1}^{l_n}k_{ni}^2k_{nj}^2\operatorname E[\langle Z_{is},H_{is-l_n-1}\rangle_{\mathrm{HS}}\langle Z_{jt},H_{jt-l_n-1}\rangle_{\mathrm{HS}}]
\end{align*}
and
\begin{align*}
	\operatorname E[\langle Z_{is},H_{is-l_n-1}\rangle_{\mathrm{HS}}\langle Z_{jt},H_{jt-l_n-1}\rangle_{\mathrm{HS}}]
	&=\sum_{r=l_n+2}^{s-l_n-1}\sum_{v=l_n+2}^{t-l_n-1}\operatorname E[\langle Z_{is},Z_{ir}\rangle_{\mathrm{HS}}\langle Z_{jt},Z_{jv}\rangle_{\mathrm{HS}}]\\
	&=\sum_{r=l_n+2}^{s-l_n-1}\sum_{v=l_n+2}^{t-l_n-1}\operatorname E\langle u_s,\langle u_{r-i},u_{s-i}\rangle\langle u_t,u_v\rangle\langle u_{v-j},u_{t-j}\rangle u_r\rangle\\
	&=0
\end{align*}
since $t<s$.
Hence,
\begin{equation}\label{eq:sigma^2(n)}
	\sigma^2(n)
	=2\opnorm{\mathcal C(0)}_2^4n^{-2}(n-2l_n-2)(n-2l_n-1)\sum_{j=1}^{l_n}k_{nj}^4
\end{equation}
and $\sigma^2(n)/(2\sigma^8p_nD(k))\to\opnorm{\mathcal C(0)}_2^4/\sigma^8$ as $n\to\infty$.

We show that \eqref{a} holds. We have that
\[
	\operatorname E|U_{nt}|^4
	\le48\Bigl|\sum_{j=1}^{l_n}k_{nj}^4(\operatorname E|\langle Z_{jt},H_{jt-l_n-1}\rangle_{\mathrm{HS}}|^4)^{1/2}\Bigr|^2
\]
since
\begin{multline*}
	\operatorname E\Bigl|\sum_{j=1}^{l_n}k_{nj}^2\langle Z_{jt},H_{jt-l_n-1}\rangle_{\mathrm{HS}}\Bigr|^4
	=\sum_{j=1}^{l_n}k_{nj}^8\operatorname E|\langle Z_{jt},H_{jt-l_n-1}\rangle_{\mathrm{HS}}|^4\\
	+6\sum_{j=2}^{l_n}\sum_{i=1}^{j-1}k_{nj}^4k_{ni}^4\operatorname E|\langle Z_{jt},H_{jt-l_n-1}\rangle_{\mathrm{HS}}\langle Z_{it},H_{it-l_n-1}\rangle_{\mathrm{HS}}|^2
\end{multline*}
and
\begin{multline*}
	\operatorname E|\langle Z_{jt},H_{jt-l_n-1}\rangle_{\mathrm{HS}}\langle Z_{it},H_{it-l_n-1}\rangle_{\mathrm{HS}}|^2\\
	\le(\operatorname E|\langle Z_{jt},H_{jt-l_n-1}\rangle_{\mathrm{HS}}|^4)^{1/2}(\operatorname E|\langle Z_{it},H_{it-l_n-1}\rangle_{\mathrm{HS}}|^4)^{1/2}
\end{multline*}
by the Cauchy-Schwarz inequality. Again by the Cauchy-Schwarz inequality and independence of $u_t$'s,
\[
	\operatorname E|\langle Z_{jt},H_{jt-l_n-1}\rangle_{\mathrm{HS}}|^4
	\le\operatorname E|\opnorm{Z_{jt}}_2\opnorm{H_{jt-l_n-1}}_2|^4
	=(\operatorname E\|u_0\|^4)^2\operatorname E\opnorm{H_{jt-l_n-1}}_2^4.
\]
Also,
\begin{equation}\label{eq:4thmomentofH}
	\operatorname E\opnorm{H_{jt-l_n-1}}_2^4
	=\operatorname E\Big|\sum_{s=l_n+2}^{t-l_n-1}\sum_{r=l_n+2}^{t-l_n-1}\langle Z_{js},Z_{jr}\rangle_{\mathrm{HS}}\Bigr|^2=O(t^2).
\end{equation}
Hence, $\operatorname E|U_{nt}|^4=O(p_n^2t^2)$ and $\sigma^{-4}(n)n^{-4}\sum_{t=2l_n+3}^n\operatorname EU_{nt}^4=O(n^{-1})$.

Now we move to the proof of \eqref{b}. Let us denote $\ddot U_n^2=n^{-2}\sum_{t=2l_n+3}^n\ddot U_{nt}^2$. We show that $\sigma^{-2}(n)[\ddot U_{nt}^2-\sigma^2(n)]\xrightarrow{p}0$ by showing that $\sigma^{-4}(n)\operatorname E|\ddot U_n^2-\sigma^2(n)|^2\to0$ as $n\to\infty$. We have that
\begin{align*}
	\ddot U_{nt}^2
	&=4\sum_{j=1}^{l_n}k_{nj}^4\operatorname E[|\langle Z_{jt},H_{jt-l_n-1}\rangle_{\mathrm{HS}}|^2\mid F_{t-1}]\\
	&\quad+8\sum_{j=2}^{l_n}\sum_{i=1}^{j-1}k_{nj}^2k_{ni}^2\operatorname E[\langle Z_{jt},H_{jt-l_n-1}\rangle_{\mathrm{HS}}\langle Z_{it},H_{it-l_n-1}\rangle_{\mathrm{HS}}\mid F_{t-1}]\\
	&=4B_{nt}+4A_{1nt}.
\end{align*}
Since
\begin{multline*}
	\operatorname E[|\langle Z_{jt},H_{jt-l_n-1}\rangle_{\mathrm{HS}}|^2\mid F_{t-1}]
	=\sum_{s=l_n+2}^{t-l_n-1}\operatorname E[|\langle Z_{jt},Z_{js}\rangle_{\mathrm{HS}}|^2\mid F_{t-1}]\\
	+2\sum_{s=l_n+3}^{t-l_n-1}\sum_{r=l_n+2}^{s-1}\operatorname E[\langle Z_{jt},Z_{js}\rangle_{\mathrm{HS}}\langle Z_{jt},Z_{jr}\rangle_{\mathrm{HS}}\mid F_{t-1}],
\end{multline*}
\begin{align*}
	B_{nt}
	&=\sum_{j=1}^{l_n}k_{nj}^4\bigg[\sum_{s=l_n+2}^{t-l_n-1}\operatorname E[|\langle Z_{jt},Z_{js}\rangle_{\mathrm{HS}}|^2\mid F_{t-1}]\biggr]\\
	&\quad+\sum_{j=1}^{l_n}k_{nj}^4\bigg[2\sum_{s=l_n+3}^{t-l_n-1}\sum_{r=l_n+2}^{s-1}\operatorname E[\langle Z_{jt},Z_{js}\rangle_{\mathrm{HS}}\langle Z_{jt},Z_{jr}\rangle_{\mathrm{HS}}\mid F_{t-1}]\biggr]\\
	&=C_{nt}+A_{2nt}.
\end{align*}
Finally,
\begin{align*}
	C_{nt}
	&=\sum_{j=1}^{l_n}k_{nj}^4\bigg[\sum_{s=l_n+2}^{t-l_n-1}(\operatorname E[|\langle Z_{jt},Z_{js}\rangle_{\mathrm{HS}}|^2\mid F_{t-1}]-\opnorm{\mathcal C(0)}_2^4)\biggr]
	+\opnorm{\mathcal C(0)}_2^4(t-2l_n-2)\sum_{j=1}^{l_n}k_{nj}^4\\
	&=A_{3nt}+\opnorm{\mathcal C(0)}_2^4(t-2l_n-2)\sum_{j=1}^{l_n}k_{nj}^4.
\end{align*}
Hence,
\begin{align*}
	\ddot U_n^2
	&=2\opnorm{\mathcal C(0)}_2^4n^{-2}(n-2l_n-2)(n-2l_n-1)\sum_{j=1}^{l_n}k_{nj}^4+4n^{-2}\sum_{t=2l_n+3}^n\sum_{j=1}^3A_{jnt}\\
	&=\sigma^2(n)+4n^{-2}\sum_{t=2l_n+3}^n\sum_{j=1}^3A_{jnt}
\end{align*}
using \eqref{eq:sigma^2(n)} and $\ddot U_n^2-\sigma^2(n)=4n^{-2}\sum_{t=2l_n+3}^n\sum_{j=1}^3A_{jnt}$.

Now we show that $\sigma^{-4}(n)n^{-4}\operatorname E|\sum_{t=2l_n+3}^n A_{jnt}|^2\to0$ as $n\to\infty$ for $j=1,2,3$. First, we establish that $n^{-4}\operatorname E|\sum_{t=2l_n+3}^nA_{1nt}|^2=o(p_n^2)$ as $n\to\infty$. We have that
\begin{multline}\label{eq:condexp}
	\operatorname E[\langle Z_{jt},H_{jt-l_n-1}\rangle_{\mathrm{HS}}\langle Z_{it},H_{it-l_n-1}\rangle_{\mathrm{HS}}\mid F_{t-1}]\\
	=\sum_{s=l_n+2}^{t-l_n-1}\sum_{v=l_n+2}^{t-l_n-1}\langle u_{s-j},u_{t-j}\rangle\langle u_{v-i},u_{t-i}\rangle\operatorname E[\langle u_t,u_s\rangle\langle u_t,u_v\rangle\mid F_{t-1}].
\end{multline}
The largest index of $u_t$'s in \eqref{eq:condexp} is $t-i$ ($t-i>t-j$ since $j>i$ in the expression of $A_{1nt}$; $t-i$ is greater than $s-j$, $v-i$, $s$ and $v$ since $t-i\ge t- l_n+1$ while $s-j\le t-l_n-3$, $v-i\le t-l_n-2$; $s\le t-l_n-1$ and $v\le t-l_n-1$). The second largest index is $t-j$ and it does not coincide with any other indices ($t-j$ is greater than $s-j$, $v-i$, $s$ and $v$ since $t-j\ge t-l_n$ while $s-j\le t-l_n-3$, $v-i\le t-l_n-2$, $s\le t-l_n-1$ and $v\le t-l_n-1$). Together with \autoref{lemma:expofinnerprod}, this implies that
\begin{multline*}
	\operatorname E[\operatorname E[\langle Z_{j_1t},H_{j_1t-l_n-1}\rangle_{\mathrm{HS}}\langle Z_{i_1t},H_{i_1t-l_n-1}\rangle_{\mathrm{HS}}\mid F_{t-1}]\\
	\times\operatorname E[\langle Z_{j_2t},H_{j_2t-l_n-1}\rangle_{\mathrm{HS}}\langle Z_{i_2t},H_{i_2t-l_n-1}\rangle_{\mathrm{HS}}\mid F_{t-1}]]=0
\end{multline*}
if $i_1\ne i_2$ or $j_1\ne j_2$. Hence,
\[
	\operatorname EA_{1nt}^2
	=4\sum_{j=2}^{l_n}\sum_{i=1}^{j-1}k_{nj}^4k_{ni}^4\operatorname E|\operatorname E[\langle Z_{jt},H_{jt-l_n-1}\rangle_{\mathrm{HS}}\langle Z_{it},H_{it-l_n-1}\rangle_{\mathrm{HS}}\mid F_{t-1}]|^2.
\]
Using the Cauchy-Schwarz inequality twice, the fact that $u_t$'s are independent and that  $F_t$ is the $\sigma$-field consisting of $u_s$, $s\le t$,
\begin{align*}
	&\operatorname E\bigl|\operatorname E[\langle Z_{jt},H_{jt-l_n-1}\rangle_{\mathrm{HS}}\langle Z_{it},H_{it-l_n-1}\rangle_{\mathrm{HS}}\mid F_{t-1}]\bigr|^2\le\\
	&\le\operatorname E\bigl|\operatorname E[\|u_t\|\|u_{t-j}\|\opnorm{H_{jt-l_n-1}}_2\|u_t\|\|u_{t-i}\|\opnorm{H_{it-l_n-1}}_2\mid F_{t-1}]\bigr|^2\\
	&=\sigma^8\operatorname E\bigl|\opnorm{H_{jt-l_n-1}}_2\opnorm{H_{it-l_n-1}}_2\bigr|^2\\
	&\le\sigma^8(\operatorname E\opnorm{H_{jt-l_n-1}}_2^4)^{1/2}(\operatorname E\opnorm{H_{it-l_n-1}}_2^4)^{1/2}.
\end{align*}
Using \eqref{eq:4thmomentofH}, we conclude that $\operatorname EA_{1nt}^2=O(p_n^2t^2)$. Since the largest index of $u_t$'s in $A_{1nt}$ is $t-i$, $\operatorname E[A_{1nt}A_{1ns}]=0$ if $t-l_n>s-1$. Hence, using the Cauchy-Schwarz inequality and the fact that $\sum_{|t-s|\le l_n}ts=O(l_nn^3)$,
\[
	\operatorname E\biggl|\sum_{t=2l_n+3}^nA_{1nt}\biggr|^2
	=\sum_{|t-s|\le l_n}\operatorname E[A_{1nt}A_{1ns}]
	\le \sum_{|t-s|\le l_n}(\operatorname EA_{1nt}^2)^{1/2}(\operatorname EA_{1ns}^2)^{1/2}
	=O(p_n^2l_nn^3)
\]
and, since $l_n/n\to0$, $n^{-4}\operatorname E|\sum_{t=2l_n+3}^nA_{1nt}|^2=o(p_n^2)$ as $n\to\infty$.

We establish that $n^{-4}\operatorname E|\sum_{t=2l_n+3}^nA_{2nt}|^2=O(p_n)$ as $n\to\infty$. Let us observe that
\begin{multline}\label{eq:A_2nt^2}
	\operatorname EA_{2nt}^2
	=\sum_{j=1}^{l_n}k_{nj}^8\operatorname E\bigg|2\sum_{s=l_n+3}^{t-l_n-1}\sum_{r=l_n+2}^{s-1}\operatorname E[\langle Z_{jt},Z_{js}\rangle_{\mathrm{HS}}\langle Z_{jt},Z_{jr}\rangle_{\mathrm{HS}}\mid F_{t-1}]\biggr|^2\\
	+2\sum_{j_1=2}^{l_n}\sum_{j_2=1}^{j_1-1}k_{nj_1}^4k_{nj_2}^4\operatorname E\biggl(\bigg[2\sum_{s_1=l_n+3}^{t-l_n-1}\sum_{r_1=l_n+2}^{s_1-1}\operatorname E[\langle Z_{j_1t},Z_{j_1s_1}\rangle_{\mathrm{HS}}\langle Z_{j_1t},Z_{j_1r_1}\rangle_{\mathrm{HS}}\mid F_{t-1}]\biggr]\\\times\bigg[2\sum_{s_2=l_n+3}^{t-l_n-1}\sum_{r_2=l_n+2}^{s_2-1}\operatorname E[\langle Z_{j_2t},Z_{j_2s_2}\rangle_{\mathrm{HS}}\langle Z_{j_2t},Z_{j_2r_2}\rangle_{\mathrm{HS}}\mid F_{t-1}]\biggr]\biggr).
\end{multline}
We have that
\[
	\operatorname E[\langle Z_{jt},Z_{js}\rangle_{\mathrm{HS}}\langle Z_{jt},Z_{jr}\rangle_{\mathrm{HS}}\mid F_{t-1}]
	=\langle u_{s-j},u_{t-j}\rangle\langle u_{r-j},u_{t-j}\rangle\operatorname E[\langle u_t,u_s\rangle\langle u_t,u_r\rangle\mid F_{t-1}],
\]
where $t-j>s>s-j\lessgtr r>r-j$. Consequently, using \autoref{lemma:expofinnerprod},
\begin{multline*}
	\operatorname E\bigl(\langle u_{s_1-j},u_{t-j}\rangle\langle u_{r_1-j},u_{t-j}\rangle\operatorname E[\langle u_t,u_{s_1}\rangle\langle u_t,u_{r_1}\rangle\mid F_{t-1}]\\
	\times\langle u_{s_2-j},u_{t-j}\rangle\langle u_{r_2-j},u_{t-j}\rangle\operatorname E[\langle u_t,u_{s_2}\rangle\langle u_t,u_{r_2}\rangle\mid F_{t-1}]\bigr)=0
\end{multline*}
if $r_1\ne r_2$ or $s_1\ne s_2$. It follows that
\begin{multline*}
	\operatorname E\bigg|2\sum_{s=l_n+3}^{t-l_n-1}\sum_{r=l_n+2}^{s-1}\operatorname E[\langle Z_{jt},Z_{js}\rangle_{\mathrm{HS}}\langle Z_{jt},Z_{jr}\rangle_{\mathrm{HS}}\mid F_{t-1}]\biggr|^2\\
	=4\sum_{s=l_n+3}^{t-l_n-1}\sum_{r=l_n+2}^{s-1}\operatorname E|\operatorname E[\langle Z_{jt},Z_{js}\rangle_{\mathrm{HS}}\langle Z_{jt},Z_{jr}\rangle_{\mathrm{HS}}\mid F_{t-1}]|^2	
	\le4\sigma^8\mu^2{t-2l_n-2\choose2}
	=O(t^2)
\end{multline*}
as $t\to\infty$. We now investigate
\begin{multline}\label{eq:expA_nt2}
	\operatorname E\bigl[\langle u_{s_1-j_1},u_{t-j_1}\rangle\langle u_{r_1-j_1},u_{t-j_1}\rangle\operatorname E[\langle u_t,u_{s_1}\rangle\langle u_t,u_{r_1}\rangle\mid F_{t-1}]\\
	\times\langle u_{s_2-j_2},u_{t-j_2}\rangle\langle u_{r_2-j_2},u_{t-j_2}\rangle\operatorname E[\langle u_t,u_{s_2}\rangle\langle u_t,u_{r_2}\rangle\mid F_{t-1}]\bigr],
\end{multline}
where $t-j_1>s_1>s_1-j_1\lessgtr r_1>r_1-j_1$ and $t-j_2>s_2>s_2-j_2\lessgtr r_2>r_2-j_2$. If  $r_1-j_1\ne r_2-j_2$, the expected value in \eqref{eq:expA_nt2} is equal to $0$. Thus, we assume that $r_1-j_1=r_2-j_2$ (this implies that $r_1>r_2$ since $j_1>j_2$). Also, $r_2=r_1-j_1+j_2>r_1-j_1$. Hence, the expected value is equal to $0$ if  $r_2\ne s_2-j_2$. We assume that $r_2=s_2-j_2$ ($s_2=r_2+j_2=r_1-j_1+2j_2$). Since $t-j_2\ge t-l_n+1$ and $r_1\le t-l_n-2$, $r_1=r_1-j_1+2j_2$, which implies that $j_1=2j_2$ or $r_1=s_1-j_1$. Otherwise the expected in \eqref{eq:expA_nt2} value is equal to $0$. We conclude that the second term on the right hand side of \eqref{eq:A_2nt^2} is $O(p_n^2t+p_nt^2)$, $\operatorname EA_{2nt}^2=O(p_nt^2+p_n^2t)$ and, by Minkowski's inequality,
\[
	n^{-4}\operatorname E\biggl|\sum_{t=2l_n+3}^nA_{2nt}\biggr|^2
	\le n^{-4}\operatorname E\biggl|\sum_{t=2l_n+3}^n(\operatorname EA_{2nt}^2)^{1/2}\biggr|^2
	=O(p_n)
	\quad\text{as}\quad n\to\infty.
\]

Lastly, we show that $n^{-4}\operatorname E|\sum_{t=2l_n+3}^nA_{3nt}|^2=o(p_n^2)$ as $n\to\infty$. We have that
\begin{equation}\label{eq:A_3nt^2}
	n^{-4}\operatorname E\bigg|\sum_{t=2l_n+3}^nA_{3nt}\biggr|^2
	=n^{-4}\sum_{t=2l_n+3}^n\operatorname EA_{3nt}^2
	+2n^{-4}\sum_{t_1=2l_n+4}^n\sum_{t_2=2l_n+3}^{t_1-1}\operatorname E[A_{3nt_1}A_{3nt_2}].
\end{equation}
Since $\operatorname E[|\langle Z_{jt},Z_{js}\rangle_{\mathrm{HS}}|^2\mid F_{t-1}]=\opnorm{\mathcal C(0)}_2^4$ (see \autoref{lemma:HSsq} below), we obtain
\begin{multline*}
	\operatorname EA_{3nt}^2
	=\sum_{j=1}^{l_n}k_{nj}^8\operatorname E\bigg|\sum_{s=l_n+2}^{t-l_n-1}(\operatorname E[|\langle Z_{jt},Z_{js}\rangle_{\mathrm{HS}}|^2\mid F_{t-1}]-\opnorm{\mathcal C(0)}_2^4)\biggr|^2
	+2\sum_{j_1=2}^{l_n}\sum_{j_2=1}^{j_1-1}k_{nj_1}^4k_{nj_2}^4\\
	\times\sum_{s_1=l_n+2}^{t-l_n-1}\sum_{s_2=l_n+2}^{t-l_n-1}\operatorname{Cov}[\operatorname E[|\langle Z_{j_1t},Z_{j_1s_1}\rangle_{\mathrm{HS}}|^2\mid F_{t-1}],\operatorname E[|\langle Z_{j_2t},Z_{j_2s_2}\rangle_{\mathrm{HS}}|^2\mid F_{t-1}]]
\end{multline*}
and
\begin{multline*}
	\operatorname E\bigg|\sum_{s=l_n+2}^{t-l_n-1}(\operatorname E[|\langle Z_{jt},Z_{js}\rangle_{\mathrm{HS}}|^2\mid F_{t-1}]-\opnorm{\mathcal C(0)}_2^4)\biggr|^2
	=\sum_{s=l_n+2}^{t-l_n-1}\operatorname{Var}[\operatorname E[|\langle Z_{jt},Z_{js}\rangle_{\mathrm{HS}}|^2\mid F_{t-1}]]\\
	\qquad+2\sum_{s=l_n+3}^{t-l_n-1}\sum_{r=l_n+2}^{s-1}\operatorname{Cov}[\operatorname E[|\langle Z_{jt},Z_{js}\rangle_{\mathrm{HS}}|^2\mid F_{t-1}],\operatorname E[|\langle Z_{jt},Z_{jr}\rangle_{\mathrm{HS}}|^2\mid F_{t-1}]].
\end{multline*}
Using the Cauchy-Schwarz inequality, the fact that $u_t$ is independent of $F_{t-1}$, $u_{t-j}$, $u_s$ and $u_{s-j}$ are $F_{t-1}$-measurable and the independence of $u_t$'s (let us observe that $t-j\ge t-l_n$ and $s\le t-l_n-1$),
\[
	\operatorname E|\operatorname E[|\langle Z_{jt},Z_{js}\rangle_{\mathrm{HS}}|^2\mid F_{t-1}]|^2
	\le\operatorname E|\operatorname E[\|u_t\|^2\|u_{t-j}\|^2\|u_s\|^2\|u_{s-j}\|^2\mid F_{t-1}]|^2
	=\sigma^4\mu_4^3<\infty.
\]
Since
\begin{equation}\label{eq:varbound}
	\operatorname{Var}[\operatorname E[|\langle Z_{jt},Z_{js}\rangle_{\mathrm{HS}}|^2\mid F_{t-1}]]
	\le\operatorname E|\operatorname E[|\langle Z_{jt},Z_{js}\rangle_{\mathrm{HS}}|^2\mid F_{t-1}]|^2\\
	\le\sigma^4\mu_4^3
\end{equation}
we obtain $\sum_{s=l_n+2}^{t-l_n-1}\operatorname{Var}[\operatorname E[|\langle Z_{jt},Z_{js}\rangle_{\mathrm{HS}}|^2\mid F_{t-1}]]=O(t)$ as $t\to\infty$. Using the fact that
\begin{multline*}
	|\operatorname{Cov}[\operatorname E[|\langle Z_{jt},Z_{js}\rangle_{\mathrm{HS}}|^2\mid F_{t-1}],\operatorname E[|\langle Z_{jt},Z_{jr}\rangle_{\mathrm{HS}}|^2\mid F_{t-1}]]|\\
	\le(\operatorname{Var}[\operatorname E[|\langle Z_{jt},Z_{js}\rangle_{\mathrm{HS}}|^2\mid F_{t-1}]])^{1/2}(\operatorname{Var}[\operatorname E[|\langle Z_{jt},Z_{jr}\rangle_{\mathrm{HS}}|^2\mid F_{t-1}]])^{1/2}
	\le\sigma^4\mu_4^3,
\end{multline*}
we obtain
\begin{multline*}
	2\sum_{s=l_n+3}^{t-l_n-1}\sum_{r=l_n+2}^{s-1}|\operatorname{Cov}[\operatorname E[|\langle Z_{jt},Z_{js}\rangle_{\mathrm{HS}}|^2\mid F_{t-1}],\operatorname E[|\langle Z_{jt},Z_{jr}\rangle_{\mathrm{HS}}|^2\mid F_{t-1}]]|\\
	\le2\sigma^4\mu_4^3\sum_{s=l_n+3}^{t-l_n-1}(s-l_n-2)=\sigma^4\mu_4^3(t-2l_n-2)(t-2l_n-3)=O(t^2)
\end{multline*}
as $t\to\infty$. We have that
\begin{align*}
	&2\sum_{j_1=2}^{l_n}\sum_{j_2=1}^{j_1-1}k_{nj_1}^4k_{nj_2}^4
	\sum_{s_1=l_n+2}^{t-l_n-1}\sum_{s_2=l_n+2}^{t-l_n-1}
	\operatorname{Cov}(
	\operatorname E[|\langle Z_{j_1t},Z_{j_1s_1}\rangle_{\mathrm{HS}}|^2\mid F_{t-1}],
	\operatorname E[|\langle Z_{j_2t},Z_{j_2s_2}\rangle_{\mathrm{HS}}|^2\mid F_{t-1}])=\\
	&=O(p_n^2t^2)	
\end{align*}
since
\begin{multline*}
	\operatorname{Cov}(
	\operatorname E[|\langle Z_{j_1t},Z_{j_1s_1}\rangle_{\mathrm{HS}}|^2\mid F_{t-1}],
	\operatorname E[|\langle Z_{j_2t},Z_{j_2s_2}\rangle_{\mathrm{HS}}|^2\mid F_{t-1}])\\
	\le(\operatorname{Var}[\operatorname E[|\langle Z_{j_1t},Z_{j_1s_1}\rangle_{\mathrm{HS}}|^2\mid F_{t-1}]])^{1/2}(\operatorname{Var}[\operatorname E[|\langle Z_{j_2t},Z_{j_2s_2}\rangle_{\mathrm{HS}}|^2\mid F_{t-1}]])^{1/2}
	\le\sigma^4\mu_4^3.
\end{multline*}
Now we bound the second term on the right hand side of \eqref{eq:A_3nt^2}. We have that
\begin{multline*}
	\operatorname E[A_{3nt_1}A_{3nt_2}]
	=\sum_{j_1=1}^{l_n}k_{nj_1}^4\sum_{j_2=1}^{l_n}k_{nj_2}^4\sum_{s_1=l_n+2}^{t_1-l_n-1}\sum_{s_2=l_n+2}^{t_2-l_n-1}\\
	\times\operatorname{Cov}[\operatorname E[|\langle Z_{j_1t_1},Z_{j_1s_1}\rangle_{\mathrm{HS}}|^2\mid F_{t_1-1}],\operatorname E[|\langle Z_{j_2t_2},Z_{j_2s_2}\rangle_{\mathrm{HS}}|^2\mid F_{t_2-1}]],
\end{multline*}
where $t_1>t_2$, and $\operatorname E[|\langle Z_{j_1t_1},Z_{j_1s_1}\rangle_{\mathrm{HS}}|^2\mid F_{t_1-1}]=|\langle u_{s_1-j_1},u_{t_1-j_1}\rangle|^2\operatorname E[|\langle u_{t_1},u_{s_1}\rangle|^2\mid F_{t_1-1}]$. Let us investigate
\[
	\operatorname{Cov}[|\langle u_{s_1-j_1},u_{t_1-j_1}\rangle|^2\operatorname E[|\langle u_{t_1},u_{s_1}\rangle|^2\mid F_{t_1-1}],
	|\langle u_{s_2-j_2},u_{t_2-j_2}\rangle|^2\operatorname E[|\langle u_{t_2},u_{s_2}\rangle|^2\mid F_{t_2-1}]],
\]
where $t_1>t_2$. Let us observe that $t_1-j_1\ne s_2-j_2$ since $t_1-j_1\ge t_1-l_n$ and $s_2-j_2\le t_2-l_n-2<t_1-l_n-2$ ($t_1>t_2$). Also, $t_1-j_1\ne s_2$ since $t_1-j_1\ge t_1-l_n$ and $s_2\le t_2-l_n-1\le t_1-l_n-2$.
Thus, the covariance is equal to $0$ unless at least one of the following conditions is true:
\begin{inparaenum}[(i)]
\item$s_1-j_1=s_2-j_2$;
\item$s_1-j_1=t_2-j_2$;
\item$s_1-j_1=s_2$;
\item$t_1-j_1=t_2-j_2$;
\item$s_1=s_2-j_2$;
\item$s_1=t_2-j_2$;
\item$s_1=s_2$.
\end{inparaenum}
If any of these conditions is true, then $2n^{-4}\sum_{t_1=2l_n+4}^n\sum_{t_2=2l_n+3}^{t_1-1}\operatorname E[A_{3nt_1}A_{3nt_2}]$ is either $O(p_n)$ or $O(p_n^2/n)$ as $n\to\infty$ and we conclude that $n^{-4}\operatorname E|\sum_{t=2l_n+3}^nA_{3nt}|^2=o(p_n^2)$ as $n\to\infty$. This completes the proof of \autoref{thm:main}.
\end{proof}

\subsection{Auxilliary lemmas}
We prove two auxiliary lemmas that are used in the proof of \autoref{thm:main}. Let us recall that $Z_{jt}=u_t\otimes u_{t-j}$ and $w_{jts}=2\langle Z_{jt},Z_{js}\rangle_{\mathrm{HS}}$.
\begin{lemma}\label{lemma:HSsq}
Suppose that the conditions of \autoref{thm:main} hold. Suppose that $1\le j\le l_n$, $2l_n+3\le t\le n$ and $l_n+2\le s\le t-l_n-1$, where $\{l_n\}_{n\ge1}$ is a real sequence. Then
\[
	\operatorname E[\operatorname E[|\langle Z_{jt},Z_{js}\rangle_{\mathrm{HS}}|^2\mid F_{t-1}]]
	=\opnorm{\mathcal C(0)}_2^4
\]
\end{lemma}
\begin{proof}
Since $\langle Z_{jt},Z_{js}\rangle_{\mathrm{HS}}=\langle u_t,u_s\rangle\langle u_{s-j},u_{t-j}\rangle$ and $\langle u_{s-j},u_{t-j}\rangle$ is $F_{t-1}$-measurable, we obtain
\[
	\operatorname E[\operatorname E[|\langle Z_{jt},Z_{js}\rangle_{\mathrm{HS}}|^2\mid F_{t-1}]]
	=\opnorm{\mathcal C(0)}_2^2\operatorname E[\operatorname E[|\langle u_t,u_s\rangle|^2\mid F_{t-1}]]	
\]
using the independence of $u_t$'s (we have that $t-j\ge t-l_n$ and $s\le t-l_n-1$) and \autoref{lemma:expofinnerprodsq}. Let $\{ e_n\}_{n\ge1}$ and  be an orthonormal basis of $\mathbb H$. Then
\begin{align*}
	\operatorname E[\operatorname E[|\langle u_t,u_s\rangle|^2\mid F_{t-1}]]
	&=\operatorname E\Bigl[\operatorname E\Big[\sum_{n=1}^\infty\langle u_t,e_n\rangle\langle e_n,u_s\rangle\sum_{m=1}^\infty\langle u_s,e_m\rangle\langle e_m,u_t\rangle\mid F_{t-1}\Bigr]\Bigr]\\
	&=\operatorname E\Bigl[\sum_{n=1}^\infty\sum_{m=1}^\infty\operatorname E[\langle u_t,e_n\rangle\langle e_n,u_s\rangle\langle u_s,e_m\rangle\langle e_m,u_t\rangle\mid F_{t-1}]\Bigr]\\
	&=\operatorname E\Bigl[\sum_{n=1}^\infty\sum_{m=1}^\infty\langle e_n,u_s\rangle\langle u_s,e_m\rangle\operatorname E[\langle u_t,e_n\rangle\langle e_m,u_t\rangle]\Bigr]\\
	&=\sum_{n=1}^\infty\sum_{m=1}^\infty\operatorname E[\langle e_n,u_0\rangle\langle u_0,e_m\rangle]\operatorname E[\langle u_0,e_n\rangle\langle e_m,u_0\rangle]\\
	&=\opnorm{\mathcal C(0)}_2^2
\end{align*}
by Parseval's identity and the definition of the Hilbert-Schmidt norm. The proof is complete.
\end{proof}

\begin{lemma}\label{lemma:expectation0}
Suppose that the conditions of \autoref{thm:main} hold. If $t_1>s_1$ and $t_2>s_2$ and $1\le j_1,j_2\le n-2$, then
\[
	\operatorname E[w_{j_1t_1s_1}w_{j_2t_2s_2}]=0
\]
unless the following two conditions hold:
\begin{inparaenum}[(a)]
\item$t_1=t_2$
\item\label{cond:exp_b}$j_1=j_2$ and $s_1=s_2$ or $s_1=t_1-j_2$ and $s_2=t_1-j_1$.
\end{inparaenum}
\end{lemma}
\begin{proof}
We start by noticing that
\[
	\operatorname E[w_{j_1t_1s_1}w_{j_2t_2s_2}]
	=4\operatorname E[\langle u_{t_1},u_{s_1}\rangle\langle u_{s_1-j_1},u_{t_1-j_1}\rangle\langle u_{t_2},u_{s_2}\rangle\langle u_{s_2-j_2},u_{t_2-j_2}\rangle],
\]
where $t_1>t_1-j_1\lessgtr s_1>s_1-j_1$ and $t_2>t_2-j_2\lessgtr s_2>s_2-j_2$. It follows that $\operatorname E[w_{j_1t_1s_1}w_{j_2t_2s_2}]=0$ if $t_1\ne t_2$ or $s_1-j_1\ne s_2-j_2$ by \autoref{lemma:expofinnerprod}. Since \eqref{cond:exp_b} implies that $s_1-j_1=s_2-j_2$, we show that $\operatorname E[w_{j_1t_1s_1}w_{j_2t_2s_2}]=0$ if \eqref{cond:exp_b} does not hold.

Let us assume that $t_1=t_2$ and $s_1-j_1=s_2-j_2$ (otherwise $\operatorname E[w_{j_1t_1s_1}w_{j_2t_2s_2}]=0$). Any two of the conditions $j_1=j_2$, $s_1=s_2$ and $s_1-j_1=s_2-j_2$ imply the third one. Similarly, any two of the conditions $s_1=t_1-j_2$, $s_2=t_1-j_1$ and $s_1-j_1=s_2-j_2$ imply the third one. Hence, we only need to consider the case, when $j_1\ne j_2$, $s_1\ne s_2$, $s_1\ne t_1-j_2$ and $s_2\ne t_1-j_1$. Since $s_2=s_1-(j_1-j_2)$, we have that
\begin{align*}
	t_1&>t_1-j_1\lessgtr s_1>s_1-j_1,\\
	t_1&>t_1-j_2\lessgtr s_1-(j_1-j_2)>s_1-j_1.
\end{align*}
If $s_1$ is not equal to any other index of $u_t$'s, $\operatorname E[w_{j_1t_1s_1}w_{j_2t_2s_2}]=0$ by \autoref{lemma:expofinnerprod}. The only option is $s_1=t_1-j_1$ and $t_1-j_2=s_1-(j_1-j_2)$ since we assume that $j_1\ne j_2$ and $s_1\ne t_1-j_2$. But both of these equalities cannot be true since we assume that $j_1\ne j_2$. This shows that $\operatorname E[w_{j_1t_1s_1}w_{j_2t_2s_2}]=0$ if $j_1\ne j_2$, $s_1\ne s_2$, $s_1\ne t_1-j_2$ and $s_2\ne t_1-j_1$. The proof is complete.
\end{proof}

\section{Auxiliary lemma for power transformation}\label{sec:app:pt}
Here we establish the relationship between the trace of the operator $\operatorname E[X\otimes Y\tilde\otimes X\otimes Y]$ and the trace of the operator $\operatorname E[X\otimes X]$, where $X$ and $Y$ are two independent and identically distributed random elements with values in the separable Hilbert space $\mathbb H$.
\begin{lemma}\label{lemma:covcov}
Let us suppose that $X$ and $Y$ are independent and identically distributed random elements with values in the separable Hilbert space $\mathbb H$ such that $\operatorname EX=0$ and $\operatorname E\|X\|^4<\infty$. Let us denote $\mathcal A=\operatorname E[X\otimes Y\tilde\otimes X\otimes Y]$. Then
\[
	\operatorname{Tr}\mathcal A
	=[\operatorname{Tr}(\operatorname E[X\otimes X])]^2,
	\quad
	\operatorname{Tr}\mathcal A^2
	=[\operatorname{Tr}(\operatorname E[X\otimes X]^2)]^2
	\quad\text{and}\quad
	\operatorname{Tr}\mathcal A^3
	=[\operatorname{Tr}(\operatorname E[X\otimes X]^3)]^2,
\]
where $A^k$ denotes the $k$-fold compositon of the operator $A$ with itself.
\end{lemma}
\begin{proof}
Let us suppose that $\{v_k\}_{k\ge1}$ are the eigenvectors of the covariance operator $\operatorname E[X\otimes X]$ and $\{\lambda_k\}_{k\ge1}$ are the corresponding eigenvalues. Also, let us denote $X_k=\langle X,v_k\rangle$, $\overline X_k=\langle v_k,X\rangle$, $Y_k=\langle Y,v_k\rangle$, $\overline Y_k=\langle v_k,Y\rangle$ for $k\ge1$. We have that $\operatorname E[X_i\overline X_j]=\lambda_i$ if $i=j$ and $0$ otherwise.
We have that
\[
	\operatorname{Tr}\mathcal A
	=\sum_{i,j\ge1}\langle\operatorname E[\overline X_iY_j(X\otimes Y)],e_i\otimes e_j\rangle_{\mathrm{HS}}
	=\sum_{i,j\ge1}\operatorname E|X_i|^2\operatorname E|Y_j|^2
	=[\operatorname{Tr}(\operatorname E[X\otimes X])]^2
\]
and
\begin{align*}
	\operatorname{Tr}\mathcal A^2
	&=\sum_{i,j\ge1}\langle\mathcal A\operatorname E\{\overline X_iY_j(X\otimes Y)\},v_i\otimes v_j\rangle_{\mathrm{HS}}\\
	&=\sum_{i,j\ge1}\operatorname E[\langle\operatorname E\{\overline X_iY_j(X\otimes Y)\},X\otimes Y\rangle_{\mathrm{HS}}X_i\overline Y_j]\\
	&=\sum_{i,j,k,l\ge1}\operatorname E[\langle\operatorname E\{\overline X_iY_j(X\otimes Y)\},v_k\otimes v_l\rangle\langle v_k\otimes v_l,X\otimes Y\rangle_{\mathrm{HS}} X_i\overline Y_j]\\
	&=\sum_{i,j,k,l\ge1}\operatorname E[\overline X_iY_jX_k\overline Y_l]\operatorname E[X_i\overline Y_j\overline X_kY_l]\\
	&=\sum_{i,j,k,l\ge1}|\operatorname E[\overline X_iX_k]|^2|\operatorname E[Y_j\overline Y_l]|^2\\
	&=\Bigl(\sum_{i\ge1}\lambda_i^2\Bigr)^2\\
	&=[\operatorname{Tr}(\operatorname E[X\otimes X]^2)]^2.
\end{align*}
Similarly,
\begin{align*}
	\operatorname{Tr}\mathcal A^3
	&=\sum_{i,j,k,l,u,v\ge1}\operatorname E[\overline X_kX_u]\operatorname E[Y_l\overline Y_v]\operatorname E[\overline X_uX_i]\operatorname E[Y_v\overline Y_j]\operatorname E[\overline X_iX_k]\operatorname E[Y_j\overline Y_l]\\
	&=\sum_{i,j\ge1}(\operatorname E|X_i|^2)^3(\operatorname E|Y_j|^2)^3\\
	&=\Bigl(\sum_{i\ge1}\lambda_i^3\Bigr)^2\\
	&=[\operatorname{Tr}(\operatorname E[X\otimes X]^3)]^2.
\end{align*}
The proof is complete.
\end{proof}

\section{Proof of  \autoref{thm:consistency}}
The bound on the estimators of the autocovariance operators in the following lemma is used in the proof of \autoref{thm:consistency}.
\begin{lemma}\label{lemma:cov_con}
Suppose that $\{u_t\}_{t\in\mathbb Z}$ satisfies \autoref{assumption:dependentseq}. Then
\[
	\sup_{j\in\mathbb Z}\operatorname E\opnorm{\hat{\mathcal C}_n(j)-\mathcal C(j)}_2^2=O(n^{-1})
	\quad\text{as}\quad n\to\infty.
\]
\end{lemma}
\begin{proof}
Using the triangle inequality and the inequality $(a+b)^2\le2a^2+2b^2$ for $a,b\in\mathbb R$,
\begin{align*}
	\opnorm{\hat{\mathcal C}_n(j)-\mathcal C(j)}_2^2
	&=\opnormBig{n^{-1}\sum_{t=j+1}^n[u_t\otimes u_{t-j}-\mathcal C(j)]-n^{-1}j\mathcal C(j)}_2^2\\
	&\le\Bigl[\opnormBig{n^{-1}\sum_{t=j+1}^n[u_t\otimes u_{t-j}-\mathcal C(j)]}_2+n^{-1}|j|\opnorm{\mathcal C(j)}_2\Bigr]^2\\
	&\le2\opnormBig{n^{-1}\sum_{t=j+1}^n[u_t\otimes u_{t-j}-\mathcal C(j)]}_2^2+2n^{-2}|j|^2\opnorm{\mathcal C(j)}_2^2.
\end{align*}
We have that $\sup_{j\in\mathbb Z}\{|j|\opnorm{\mathcal C(j)}_2^2\}<\infty$ since $\sum_{h=-\infty}^\infty\opnorm{\mathcal C(h)}_2^2<\infty$. Hence, $2n^{-2}|j|^2\opnorm{\mathcal C(j)}_2^2=O(n^{-1})$ as $n\to\infty$ because $|j|<n$. Suppose that $\{e_k\}_{k\ge1}$ and $\{f_k\}_{k\ge1}$ are two orthonormal bases of $\mathbb H$. Using the definition of the Hilbert-Schmidt norm and Parseval's identity,
\begin{align*}
	\operatorname E\opnormBig{\sum_{t=j+1}^n[u_t\otimes u_{t-j}-\mathcal C(j)]}_2^2
	&=\sum_{k=1}^\infty\operatorname E\Bigl\|\sum_{t=j+1}^n[u_t\otimes u_{t-j}-\mathcal C(j)](e_k)\Bigr\|^2\\
	&=\sum_{k=1}^\infty\operatorname E\Bigl\|\sum_{t=j+1}^n[\langle e_k,u_{t-j}\rangle u_t-\operatorname E[\langle e_k,u_0\rangle u_j]\Bigr\|^2\\
	&=\sum_{k,l=1}^\infty\operatorname E\Bigl|\sum_{t=j+1}^n[\langle e_k,u_{t-j}\rangle\langle u_t,f_l\rangle-\operatorname E[\langle e_k,u_0\rangle\langle u_j,f_l\rangle]\Bigr|^2\\
	&=\sum_{k,l=1}^\infty\sum_{s,t=j+1}^n\operatorname{Cov}[\langle e_k,u_{s-j}\rangle\langle u_s,f_l\rangle,\langle e_k,u_{t-j}\rangle\langle u_t,f_l\rangle].
\end{align*}
Since $u_t$'s have zero means, we obtain
\begin{align*}
	\operatorname{Cov}[\langle e_k,u_{s-j}\rangle\langle u_s,f_l\rangle,\langle e_k,u_{t-j}\rangle\langle u_t,f_l\rangle]
	&=\operatorname{cum}[\langle e_k,u_{s-j}\rangle,\langle u_s,f_l\rangle,\langle e_k,u_{t-j}\rangle,\langle u_t,f_l\rangle]\\
	&+\operatorname E[\langle e_k,u_{s-j}\rangle\langle e_k,u_{t-j}\rangle]\operatorname E[\langle u_s,f_l\rangle\langle u_t,f_l\rangle]\\
	&+\operatorname E[\langle e_k,u_{s-j}\rangle\langle u_t,f_l\rangle]\operatorname E[\langle u_s,f_l\rangle\langle e_k,u_{t-j}\rangle]	
\end{align*}
(see \cite[p.~34]{rosenblatt:1985}). Using the fact that $u_t$'s are fourth order stationary and \eqref{eq:cumoper},
\begin{align*}
	&n^{-1}\sum_{k,l=1}^\infty\sum_{s,t=j+1}^n\operatorname{cum}[\langle e_k,u_{s-j}\rangle,\langle u_s,f_l\rangle,\langle e_k,u_{t-j}\rangle,\langle u_t,f_l\rangle]=\\
	&=\sum_{k,l=1}^\infty\sum_{h=-(n-1)}^{n+1}(1-|h|/n)\operatorname{cum}[\langle e_k,u_0\rangle,\langle u_j,f_l\rangle,\langle e_k,u_h\rangle,\langle u_{h+j},f_l\rangle]\\
	&=\sum_{h=-(n-1)}^{n-1}(1-|h|/n)\sum_{k,l=1}^\infty\langle\mathcal K_{h+j,h,j}(f_l\otimes e_k),f_l\otimes e_k\rangle\\
	&\le\sum_{h=-(n-1)}^{n-1}(1-|h|/n)\opnorm{\mathcal K_{h+j,h,j}}_1,
\end{align*}
\begin{align*}
	&n^{-1}\sum_{k,l=1}^\infty\sum_{s,t=j+1}^n\operatorname E[\langle e_k,u_{s-j}\rangle\langle e_k,u_{t-j}\rangle]\operatorname E[\langle u_s,f_l\rangle\langle u_t,f_l\rangle]=\\
	&=\sum_{k,l=1}^\infty\sum_{h=-(n-1)}^{n-1}(1-|h|/n)\operatorname E[\langle e_k,u_0\rangle\langle e_k,u_h\rangle]\operatorname E[\langle u_0,f_l\rangle\langle u_h,f_l\rangle]\\
	&=\sum_{h=-(n-1)}^{n-1}(1-|h|/n)\sum_{k=1}^\infty\langle(\mathcal C(h))(e_k),e_k\rangle\sum_{l=1}^\infty\langle(\mathcal C(h))(f_l),f_l\rangle\\
	&\le\sum_{h=-(n-1)}^{n-1}(1-|h|/n)\opnorm{\mathcal C(h)}_1^2
\end{align*}
and, using the Cauchy-Schwarz inequality,
\begin{align*}
	&n^{-1}\sum_{k,l=1}^\infty\sum_{s,t=j+1}^n\operatorname E[\langle e_k,u_{s-j}\rangle\langle u_t,f_l\rangle]\operatorname E[\langle u_s,f_l\rangle\langle e_k,u_{t-j}\rangle]=\\
	&=\sum_{k,l=1}^\infty\sum_{h=-(n-1)}^{n-1}(1-|h|/n)\operatorname E[\langle e_k,u_0\rangle\langle u_{h+j},f_l\rangle]\operatorname E[\langle u_0,f_l\rangle\langle e_k,u_{h-j}\rangle]\\
	&=\sum_{h=-(n-1)}^{n-1}(1-|h|/n)\sum_{k,l=1}^\infty\langle(\mathcal C(h+j))(e_k),f_l\rangle\langle e_k,\langle f_l,(\mathcal C(h-j))(f_l)\rangle\\
	&\le\sum_{h=-(n-1)}^{n-1}(1-|h|/n)\Bigl[\sum_{k,l=1}^\infty|\langle(\mathcal C(h+j))(e_k),f_l\rangle|^2\Bigr]^{1/2}\Bigl[\sum_{k,l=1}^\infty|\langle e_k,\langle f_l,(\mathcal C(h-j))(f_l)\rangle|^2\Bigr]^{1/2}\\
	&=\sum_{h=-(n-1)}^{n-1}(1-|h|/n)\opnorm{\mathcal C(h+j)}_2\opnorm{\mathcal C(h-j)}_2.
\end{align*}
Let us observe that
\[
	\sum_{h=-(n-1)}^{n-1}\opnorm{\mathcal C(h+j)}_2\opnorm{\mathcal C(h-j)}_2
	\le\sum_{h=-\infty}^\infty\opnorm{\mathcal C(h)}_2^2
	\le\sum_{h=-\infty}^\infty\opnorm{\mathcal C(h)}_1^2
	<\infty
\]
using the Cauchy-Schwarz inequality and the fact that $\opnorm{A}_2\le\opnorm{A}_1$ for a nuclear operator $A$. The proof is complete.
\end{proof}

We also establish consistency in mean integrated square error of the kernel lag-window estimator of the spectral density operator.
\begin{lemma}\label{lemma:kernel_est_consist}
Suppose that $\{u_t\}_{t\in\mathbb Z}$ satisfies \autoref{assumption:dependentseq}, \autoref{assumption:kernel} and \ref{assumption:band} hold. Then
\[
	\int_{-\pi}^\pi\operatorname E\opnorm{\hat{\mathcal F}_n(\omega)-\mathcal F(\omega)}_2^2d\omega=o(1)\quad\text{as}\quad n\to\infty.
\]
\end{lemma}
\begin{proof}
Using definitions \eqref{eq:dfn_sdoper} and \eqref{eq:dfn_sdoper_est} as well as the fact that the functions $\{e_k\}_{k\in\mathbb Z}$ defined by $e_k(x)=e^{-\mathbf ikx}$ for $k\in\mathbb Z$ and $x\in[-\pi,\pi]$ are orthonormal in $L^2[-\pi,\pi]$, we obtain
\begin{align}
	&2\pi\int_{-\pi}^\pi\operatorname E\opnorm{\hat{\mathcal F}_n(\omega)-\mathcal F(\omega)}_2^2d\omega=\notag\\
	&=(2\pi)^{-1}\int_{-\pi}^\pi\operatorname E\opnormBig{\sum_{|j|<n}k_{nj}\hat{\mathcal C}_n(j)e^{-\mathbf ij\omega}-\sum_{j\in\mathbb Z}\mathcal C(j)e^{-\mathbf ij\omega}}_2^2d\omega\notag\\
	&=\sum_{|j|<n}\operatorname E\opnorm{k_{nj}\hat{\mathcal C}_n(j)-\mathcal C(j)}_2^2-\sum_{j\ge|n|}\opnorm{\mathcal C(j)}_2^2\notag\\
	&\le2\sum_{|j|<n}k_{nj}^2\operatorname E\opnorm{\hat{\mathcal C}_n(j)-\mathcal C(j)}_2^2
	+2\sum_{|j|<n}(k_{nj}-1)^2\opnorm{\mathcal C(j)}_2^2
	-\sum_{j\ge|n|}\opnorm{\mathcal C(j)}_2^2.\label{eq:consist}
\end{align}
Since $\sup_{j\in\mathbb Z}\operatorname E\opnorm{\hat{\mathcal C}_n(j)-\mathcal C(j)}_2^2=O(n^{-1})$ (see \autoref{lemma:cov_con}) and $p_n/n\to0$, the first term in \eqref{eq:consist} is $o(1)$ as $n\to\infty$. Given that $k_{nj}-1\to0$ as $p_n\to\infty$, the second term is $o(1)$ as $n\to\infty$ using Lebesgue's dominated convergence theorem. Finally, $\sum_{j\ge|n|}\opnorm{\mathcal C(j)}_2^2=o(1)$ as $n\to\infty$ since $\sum_{j=-\infty}^\infty\opnorm{\mathcal C(j)}_2^2<\infty$. The proof is complete.
\end{proof}
Now we are ready to prove \autoref{thm:consistency}.
\begin{proof}[Proof of  \autoref{thm:consistency}]
We have that
\[
	(p_n^{1/2}/n)T_n
	=\frac{2^{-1}\hat Q_n^2-\hat\sigma_n^4n^{-1}C_n(k)}{\opnorm{\hat{\mathcal C}_n(0)}_2^2(2p_n^{-1}D_n(k))^{1/2}}.
\]
Given that $p_n\to\infty$ and $p_n/n\to0$, $n^{-1}C_n(k)\to0$ and $p_n^{-1}D_n(k)\to D(k)$ as $n\to\infty$. Thus we need to establish that $Q_n^2\xrightarrow{p}Q^2$ and $\opnorm{\hat{\mathcal C}_n(0)}_2^2\xrightarrow{p}\opnorm{\mathcal C(0)}_2^2$ as $n\to\infty$ to complete the proof. Since
\begin{multline*}
	\hat Q_n^2-Q^2
	=2\pi\int_{-\pi}^\pi\opnorm{\hat{\mathcal F}_n(\omega)-\mathcal F(\omega)}_2^2d\omega
	+\opnorm{\mathcal C(0)-\hat{\mathcal C}_n(0)}_2^2\\
	+4\pi\int_{-\pi}^\pi\operatorname{Re}\langle\hat{\mathcal F}_n(\omega)-\mathcal F(\omega),\mathcal F(\omega)\rangle_{\mathrm{HS}}d\omega
	-2\int_{-\pi}^\pi\operatorname{Re}\langle\hat{\mathcal F}_n(\omega)-\mathcal F(\omega),\hat{\mathcal C}_n(0)\rangle_{\mathrm{HS}}d\omega,
\end{multline*}
we need to show that $2\pi\int_{-\pi}^\pi\opnorm{\hat{\mathcal F}_n(\omega)-\mathcal F(\omega)}_2^2d\omega=o_p(1)$ and $\opnorm{\mathcal C(0)-\hat{\mathcal C}_n(0)}_2=o_p(1)$ as $n\to\infty$, but this follows from \autoref{lemma:kernel_est_consist} and \autoref{lemma:cov_con}. The proof is complete.
\end{proof}

\section{Data driven bandwidth selection}\label{sec:band}

We now describe a data driven approach to determine the bandwidth parameter $p_n$ given a choice of the kernel $k$. In particular, we adapt the plug-in approach of \cite{newey:west:1994} and \cite{buhlmann:1996} to the setting of spectral density operator estimation with general Hilbert space valued time series. We suppose that for some integer $q \ge 1$, the kernel function $k$ satisfies that
$$
0 < \xi = \lim_{x\to 0} |x|^{-q}(1-k(x)) < \infty.
$$
With $\mathcal F(\omega)$ defined above, let
\[
	\mathcal F^{(q)}(\omega)
	=(2\pi)^{-1}\sum_{j\in\mathbb Z}\mathcal |j|^q \mathcal C(j)e^{-\mathbf ij\omega}	
\]
denote the generalized $q$th order derivative of $\mathcal F(\omega)$. We assume $\mathcal F^{(q)}$ is well defined, which holds if, for example, there exists $r\ge q$ so that
\begin{equation*}
\sum_{j=-\infty}^\infty (1+|j|)^{q+r} \opnorm{\mathcal C(j)}_2 < \infty.
\end{equation*}
In this case, under \autoref{assumption:dependentseq}, it follows along similar lines of the proof of Theorem 2.1 of \cite{berkes:horvath:rice:2016} and as in Theorem 3.2 of \cite{panaretos:tavakoli:2013} that
\begin{multline*}
\int_{-\pi}^{\pi}\operatorname E\opnorm{\mathcal F(\omega)-\hat{\mathcal F}_n(\omega)}_2^2d\omega
	= \frac{p_n}{n}\Bigl( \int_{-\pi}^{\pi}\opnorm{ \mathcal F(\omega)}_2^2d\omega + \int_{-\pi}^{\pi}\operatorname{tr}( \mathcal F(\omega) )^2 d\omega \Bigr)\int_{-\infty}^{\infty}k^2(x)dx \\
+p_n^{-2{q}}\int_{-\pi}^{\pi}\opnorm{{\xi}\mathcal{F}^{(q)}(\omega)}_2^2 d\omega+o(p_n/n + p_n^{-2q}).
\end{multline*}
A method to select the bandwidth is to attempt to minimize as a function of the bandwidth the leading terms in this asymptotic expansion. Doing so gives
$$
p_{n,opt} = {\mathfrak{c}_q} n^{1/(2q+1)},
$$
where
$$
\mathfrak{c}_q^{(2q+1)} = \frac{ 2q \xi^2 \int_{-\pi}^{\pi}\opnorm{\mathcal{F}^{(q)}(\omega)}^2 d\omega}
{\int_{-\infty}^{\infty}k^2(x)dx[\int_{-\pi}^{\pi} \opnorm{ \mathcal F(\omega)}^2d\omega + \int_{-\pi}^{\pi}  \operatorname{tr}( \mathcal F(\omega) )^2 d\omega]}.
$$
The constant $\mathfrak{c}_q$ can be estimated as follows.

{\textbf {\textit{Bandwidth selection method for spectral density operator estimation:} }}
\begin{enumerate}

\item Compute pilot estimates of $\mathcal{F}^{(r)}$, for $r=0,q$:
\[
	\hat{\mathcal F}^{(r)}_{n,\star}(\omega)
	=(2\pi)^{-1}\sum_{|j|<n}|j|^r k_{ini}^{(r)}(j/p^{(r)}_{n,ini})\hat{\mathcal C}_n(j)e^{-ij\omega},
\]
that utilize an initial bandwidth choice $p_{n,ini}^{(r)}$, and weight function $k_{ini}^{(r)}$.
\item Estimate ${\frak c}_q$ by
$$
\hat{\mathfrak{c}}_q^{(2q+1)} = \frac{ 2q \xi^2 \int_{-\pi}^{\pi}\opnorm{\hat{\mathcal{F}}_{n,\star}^{(q)}(\omega)}^2 d\omega}
{\int_{-\infty}^{\infty}k^2(x)dx[\int_{-\pi}^{\pi} \opnorm{ \hat{\mathcal F}_{n,\star}(\omega)}^2d\omega + \int_{-\pi}^{\pi}\operatorname{tr}( \hat{\mathcal F}_{n,\star}(\omega) )^2 d\omega]}.
$$
\item Use the bandwidth
\[
	\hat{p}_{n,opt}
	=\hat{\frak c}_q n^{1/(1+2{ q})}.
\]
\end{enumerate}
In our implementation to produce the results in Section \ref{simul}, we took $k_{ini}^{(r)}$, $r=0,q$, to be the same as $k$, $p^{(0)}_{n,ini} = n^{1/(2q+1)}$, and  $p^{(q)}_{n,ini} = 4n^{1/(2q+1)}$. It is advisable to take a somewhat larger bandwidth to estimate ${\mathcal F}^{(q)}$ due to the fact that the summands are enlarged by the terms of the form $|j|^q$.
}

\end{document}